\documentclass[11pt]{amsart}

\usepackage{latexsym}
\usepackage{amsfonts}

\usepackage{pgf,tikz}
\usepackage{mathrsfs}
\usetikzlibrary{arrows}
\usepackage{color}
\usepackage{float}
\usepackage{subfigure}
\usepackage{caption}

\newcommand{\field}[1]{\mathbb{#1}}
\newcommand{\C}{\field{C}}

\newcommand{\K}{\field{K}}

\newcommand{\R}{\field{R}}

\newcommand{\rr}{\field{R}}

\newcommand{\pa}[2]{\frac{\partial #1}{\partial #2}}
\newcommand{\lra}[1]{\langle #1 \rangle}

\DeclareMathOperator{\Stab}{Stab}

\DeclareMathOperator{\id}{id}

\DeclareMathOperator{\GL}{GL}

\newcommand{\cc}{\mathbf{\mathbb{C}}}

\newtheorem{defi}{Definition}[section]

\newtheorem{lem}[defi]{Lemma}
\newtheorem{theorem}[defi]{Theorem}
\newtheorem{co}[defi]{Corollary}

\newtheorem{re}[defi]{Remark}

\subjclass[2010]{14 D 99, 14 R 99, 51 M 99}
\thanks{The  authors were partially supported by the Narodowe Centrum Nauki grant number 2015/17/B/ST1/02637}
\title[On quadratic polynomial mappings]{On quadratic polynomial mappings from the plane into the $n$ dimensional space} \makeatletter

\author{Micha{\l} Farnik,  Zbigniew Jelonek, Piotr Migus}

\address[M. Farnik]{Faculty of Mathematics and Computer Science\\
Jagiellonian University\\
{\L}ojasiewicza 6, 30-348 Krak\'ow, Poland}
\email{michal.farnik@gmail.com}

\address[Z. Jelonek]{Instytut Matematyczny\\
Polska Akademia Nauk\\
\'Sniadeckich 8, 00-656 Warszawa, Poland}
\email{najelone@cyf-kr.edu.pl}

\address[P. Migus]{Department of Mathematics and Physics\\
Kielce University of Technology\\
Al. 1000 L PP 7,25-314 Kielce, Poland}
\email{migus.piotr@gmail.com}


\date{\today}

\begin{document}

\begin{abstract}{We show that, up to linear equivalence, there are only finitely many  polynomial
quadratic mappings $F:\C^2\to\C^n$ and $F:\R^2\to\R^n$. We list all possibilities.}
\end{abstract}

\maketitle

\section{Introduction}
Let $\Omega_{\K^2}(d_1,\ldots,d_n)$ denote the space of polynomial mappings $F=(f_1,\ldots,f_n):\K^2\to\K^n$ where $\deg f_i\leq d_i$ for $1\leq i\leq n$. Let $F,G\in \Omega_{\K^2}(d_1,\ldots,d_n)$. We say that $F$ is \emph{topologically} (respectively \emph{linearly}) \emph{equivalent} to $G$ if there are homeomorphisms (respectively linear isomorphisms) $\Phi:\K^2\to\K^2$ and $\Psi:\K^n\to\K^n$ such that $F=\Psi\circ G\circ \Phi$. In the paper \cite{a-n} it was shown that there is only a finite number of different topological types of mappings in  $\Omega_{\K^2}(d_1,d_2)$. This result was generalized for $\Omega_{\K^2}(d_1,\ldots,d_n)$ by Sabbah in \cite{sabb}.  Moreover, we know (see e.g. \cite{jel2}) that there is a Zariski open dense subset $U\subset \Omega_{\K^2}(d_1,\ldots,d_n)$ such that every mapping $f\in U$ has the same, generic, topological type. If a mapping $f$  has a generic topological type then we say that $f$ is a \emph{topologically generic} mapping.  In practice it is difficult to describe the generic  topological type and other topological types  effectively. Here we focus on the simplest case of quadratic polynomial mappings and we describe linear types in $\Omega_{\K^2}(2,\ldots,2)$. We have $34$ linear types in the complex case and $41$ in the real case.
From this we can  describe all topological types in $\Omega_{\K^2}(2,\ldots,2)$. In fact in the complex case we have $18$ topological types and the topological equivalence coincides with the equivalence with respect to the group of polynomial automorphisms, i.e., topological and polynomial equivalences coincide. Moreover, in the real case we have $22$ types with respect to the polynomial equivalence.

General quadratic mappings have been studied (see \cite{agrach}, \cite{gio}, \cite{i-n}, \cite{in2}, \cite{in3}, \cite{in4}) and this subject is interesting on its own. The classification (up to isotopy) of quadratic mappings of the plane was done only for real homogeneous mappings (see Proposition 1 in \cite{agrach}). In \cite{bgv}, \cite{dgr} and \cite{npg} real quadratic mappings of the plane have been studied in relation to their dynamical behavior and equivalence with respect to diffeomorphisms, in \cite{npg} the authors obtained the classification of critical sets and their images. In our recent paper \cite{fj} we classified real and complex quadratic mappings of the plane with respect to linear equivalence. Hence from this point of view we fill a gap in the literature.

In the paper \cite{fj} it was shown that for $\K$ equal $\C$ or $\R$ the space $\Omega_{\K^2}(2,2)$ splits into a finite number of equivalence classes with respect to linear equivalence (hence also with respect to topological equivalence). Moreover, the authors provided a full classification of mappings in $\Omega_{\K^2}(2,2)$. The aim of this paper is to continue the research started in \cite{fj} and to obtain a full classification of quadratic mappings $\K^2\to\K^n$ for any integer $n>2$ (and consequently together with \cite{fj} for any $n>0$).

The main tools used in \cite{fj} were the critical set, the discriminant and the topological degree. While those are still useful in analyzing $\K^2\to\K^n$ mappings, they do not provide sufficient information to distinguish between some equivalence classes. Hence in this paper we use also the self-intersection curve and the critical space. The latter can provide useful information even if the mapping is an embedding.

Aside from the complete classification we also obtain some interesting geometric results (see the comments after the proofs of Theorems \ref{thf1} and \ref{thf2}) about quadratic mappings. For example we obtain that if $F:\C^2\to\C^n$ is quadratic, $\mu(F)=1$ and $\# F^{-1}(y)\geq 3$ for some $y\in\C^n$ then $F$ is equivalent to $(x^2+y,y^2+x,xy,0,\ldots,0)$, in particular it has three singular points. Let us recall that $\mu(F)$ denotes the number of points in $F^{-1}(F(x))$ for a generic $x\in \C^2.$ For the definition of $SI_F$ (resp. $O(F)$) see Definition 2.5
(resp. Definition 2.2).

We will now give the complete classification. We start with $\Omega_{\C^2}(2,2,2)$ where we have the following cases:

\begin{itemize}

\item[(1)] (generic case) $F_1=(x^2+y,y^2+x,xy)$ with three singular points $\left(\frac{\varepsilon}{2},\frac{\varepsilon^2}{2}\right)$ and $SI_{F_1}$ being the union of three lines $y+\varepsilon x-\varepsilon^2=0$, where $\varepsilon^3=1$ and  $\dim O(F_1)=18$. Moreover, $\mu(F_1)=1$.

\item[(2)] $F_2=(x^2+y,y^2+x,xy+\frac{1}{2}x+\frac{1}{2}y)$ with singular points $P_1=(\frac{1}{2},\frac{1}{2})$ and $P_2=(-\frac{1}{2},-\frac{1}{2})$. Moreover $SI_{F_2}=V((x-y)^2(x+y-1))$, $\dim O(F_2)=17$ and $\mu(F_2)=1$.

\item[(3)] $F_3=(x^2,y^2+x,xy)$ with singular point $P=(0,0)$, $SI_{F_3}=V(x^3)$ and $\dim O(F_{3})=16$. Moreover, $\mu(F_3)=1$.

\item[(4)] $F_{4}=(x^2,y^2,xy)$ with singular point $P=(0,0)$, $\mu(F_4)=2$ and $\dim O(F_{4})=14$.

\item[(5)] $F_5=(x^2,y^2,x+y)$ with singular point $P=(0,0)$, $SI_{F_5}=V(x+y)$ and $\dim O(F_{5})=17$. Moreover, $\mu(F_5)=1$.

\item[(6)] $F_6=(x^2+y,y^2,x)$ with $C(F_6)=\emptyset$ and $\dim O(F_{6})=16$. Moreover, $\mu(F_6)=1$.

\item[(7)] $F_7=(x^2+y,y^2+x,0)$ with $C(F_7)=\{4xy-1=0\}$ and $\dim O(F_{7})=15$. Moreover, $\mu(F_7)=4$.

\item[(8)] $F_8=(x^2,xy,y)$ with singular point $P=(0,0)$, $SI_{F_8}=V(y)$ and $\dim O(F_{8})=16$. Moreover, $\mu(F_8)=1$.

\item[(9)] $F_9=(x^2+y,xy,x)$ with $C(F_9)=\emptyset$ and $\dim O(F_{9})=15$. Moreover, $\mu(F_9)=1.$

\item[(10)] $F_{10}=(x^2+y,xy,0)$ with $C(F_{10})=\{2x^2-y=0\}$ and $\dim O(F_{10})=14$. Moreover, $\mu(F_{10})=3.$

\item[(11)] $F_{11}=(x^2,y^2,y)$ with $C(F_{11})=SI_{F_{11}}=\{x=0\}$ and $\dim O(F_{11})=15$. Moreover, $\mu(F_{11})=2.$

\item[(12)] $F_{12}=(x^2+y,y^2,0)$ with $C(F_{12})=\{4xy=0\}$ and $\dim O(F_{12})=14$. Moreover, $\mu(F_{12})=4.$

\item[(13)] $F_{13}=(x^2,y^2,0)$ with $C(F_{13})=\{4xy=0\}$ and $\dim O(F_{13})=13$. Moreover, $\mu(F_{13})=4.$

\item[(14)] $F_{14}=(x^2,xy,x)$ with $C(F_{14})=\{2x=0\}$ and $\dim O(F_{14})=14$. Moreover, $\mu(F_{14})=1.$

\item[(15)] $F_{15}=(x^2-x,xy,0)$ with $C(F_{15})=\{2x^2-x=0\}$ and $\dim O(F_{15})=13$. Moreover, $\mu(F_{15})=2.$

\item[(16)] $F_{16}=(x^2,xy,0)$ with $C(F_{16})=\{x^2=0\}$ and $\dim O(F_{16})=12$.  Moreover, $\mu(F_{16})=2.$

\item[(17)] $F_{17}=(xy,x,y)$ with $C(F_{17})=\emptyset$ and $\dim O(F_{17})=14$.  Moreover, $\mu(F_{17})=1.$

\item[(18)] $F_{18}=(x^2,x,y)$ with $C(F_{18})=\emptyset$ and $\dim O(F_{18})=13$.  Moreover, $\mu(F_{18})=1.$

\item[(19)] $F_{19}=(xy,x+y,0)$ with $C(F_{19})=\{y-x=0\}$ and $\dim O(F_{19})=13$.  Moreover, $\mu(F_{19})=2.$

\item[(20)] $F_{20}=(x,xy,0)$ with $C(F_{20})=\{x=0\}$ and $\dim O(F_{20})=12$.  Moreover, $\mu(F_{20})=1.$

\item[(21)] $F_{21}=(x^2,y,0)$ with $C(F_{21})=\{2x=0\}$ and $\dim O(F_{21})=12$.  Moreover, $\mu(F_{21})=2.$

\item[(22)] $F_{22}=(x^2+y,x,0)$ with $C(F_{22})=\emptyset$ and $\dim O(F_{22})=11$.  Moreover, $\mu(F_{22})=1.$

\item[(23)] $F_{23}=(x^2,x,0)$ with $C(F_{23})=\C^2$ and $\dim O(F_{23})=10$.

\item[(24)] $F_{24}=(x,y,0)$ with $C(F_{24})=\emptyset$ and $\dim O(F_{24})=9$.  Moreover, $\mu(F_{24})=1.$

\item[(25)] $F_{25}=(xy,0,0)$ with $C(F_{25})=\C^2$ and $\dim O(F_{25})=10$.

\item[(26)] $F_{26}=(x^2+y,0,0)$ with $C(F_{26})=\C^2$ and $\dim O(F_{26})=9$.

\item[(27)] $F_{27}=(x^2,0,0)$ with $C(F_{27})=\C^2$ and $\dim O(F_{27})=8$.

\item[(28)] $F_{28}=(x,0,0)$ with $C(F_{28})=\C^2$ and $\dim O(F_{28})=7$.

\item[(29)] $F_{29}=(0,0,0)$ with $C(F_{29})=\C^2$ and $\dim O(F_{29})=3$.

\end{itemize}

For $\Omega_{\C^2}(2,2,2,2)$ we obtain the $29$ cases $(f_1,f_2,f_3,0)$ where $(f_1,f_2,f_3)=F_i$ for $1\leq i\leq 29$. Moreover we obtain the following cases:

\begin{itemize}

\item[(1)] (generic case) $G_1=(x^2+y,y^2,xy,x)$ which is an immersion. Moreover, $\dim O(G_1)=24$.

\item[(2)] $G_2=(x^2,y^2,xy,x)$ with singular point $(0,0)$, $SI_{G_2}=\{x=0\}$ and $\dim O(G_2)=23$.  Moreover, $\mu(G_{2})=1$.

\item[(3)] $G_3=(x^2,y^2,x,y)$ which is an immersion. Moreover, $\dim O(G_3)=22$.

\item[(4)] $G_4=(x^2,xy,x,y)$ which is an immersion. Moreover, $\dim O(G_4)=21$.

\end{itemize}

Finally for quadratic mappings $\C^2\to\C^n$ for $n\geq 5$ the generic case is $G_0=(x^2,xy,y^2,x,y,0,\ldots,0)$. All other cases can be obtained from the $\Omega_{\C^2}(2,2,2,2)$ cases by composing with the standard inclusion of $\C^4$ in $\C^n$.

We also obtain similar results in the real case. For $\Omega_{\R^2}(2,2,2)$ we have the following possibilities:

\begin{itemize}
\item[(1a)] $F_1=(x^2+y,y^2+x,xy)$ with singular point $P=(\frac{1}{2},\frac{1}{2})$, $SI_{F_1}=\{(x,y):x+y-1=0\}$ and $\dim O(F_1)=18$.

\item[(1b)] $F_{1'}=(x^2-y^2+x,2xy-y,-3x^2+y^2)$ with singular points $P_1=\left( \frac{1}{2},0\right)$, $P_2=\left(\frac{1}{4},\frac{\sqrt{3}}{4} \right)$ and $P_3=\left(\frac{1}{4},\frac{-\sqrt{3}}{4} \right)$. Moreover $SI_{F_{1'}}=\{(x,y):(x-\frac{1}{2})(x^2-\frac{1}{3}y^2)=0\}$ and $\dim O(F_1)=18$.

\item[(2)] $F_2=(x^2+y,y^2+x,xy+\frac{1}{2}x+\frac{1}{2}y)$ with singular points $P_1=(\frac{1}{2},\frac{1}{2})$ and $P_2=(-\frac{1}{2},-\frac{1}{2})$. Moreover $SI_{F_2}=\{(x-y)(x+y-1)=0\}$ and $\dim O(F_2)=17$

\item[(3)] $F_3=(x^2,y^2+x,xy)$ with singular point $P=(0,0)$, $SI_{F_3}=\{x=0\}$ and $\dim O(F_{3})=16$.

\item[(4)] $F_{4}=(x^2,y^2,xy)$ with singular point $P=(0,0)$ and $\dim O(F_{4})=14$.

\item[(5)] $F_5=(x^2,y^2,x+y)$ with singular point $P=(0,0)$, $SI_{F_5}=\{x+y=0\}$ and $\dim O(F_{5})=17$.

\item[(6)] $F_6=(x^2+y,y^2,x)$ with $C(F_6)=\emptyset$ and $\dim O(F_{6})=16$.

\item[(7a)] $F_7=(x^2+y,y^2+x,0)$ with $C(F_7)=\{4xy-1=0\}$ and $\dim O(F_{7})=15$.

\item[(7b)] $F_{7'}=(x^2-y^2+x,2xy-y,0)$ with $C(F_{7'})=\{4x^2+4y^2-1=0\}$ and $\dim O(F_{7'})=15$.

\item[(8)] $F_8=(x^2,xy,y)$ with singular point $P=(0,0)$, $SI_{F_8}=\{y=0\}$ and $\dim O(F_{8})=16$.

\item[(9)] $F_9=(x^2+y,xy,x)$ with $C(F_9)=\emptyset$ and $\dim O(F_{9})=15$.

\item[(10)] $F_{10}=(x^2+y,xy,0)$ with $C(F_{10})=\{2x^2-y=0\}$ and $\dim O(F_{10})=14$.

\item[(11)] $F_{11}=(x^2,y^2,y)$ with $C(F_{11})=SI_{F_{11}}=\{x=0\}$ and $\dim O(F_{11})=15$.

\item[(12)] $F_{12}=(x^2+y,y^2,0)$ with $C(F_{12})=\{4xy=0\}$ and $\dim O(F_{12})=14$.

\item[(13a)] $F_{13}=(x^2,y^2,0)$ with $C(F_{13})=\{4xy=0\}$ and $\dim O(F_{13})=13$.

\item[(13b)] $F_{13'}=(x^2-y^2,xy,0)$ with $C(F_{13'})=\{(0,0)\}$ and $\dim O(F_{13'})=13$.

\item[(14)] $F_{14}=(x^2,xy,x)$ with $C(F_{14})=\{2x=0\}$ and $\dim O(F_{14})=14$.

\item[(15)] $F_{15}=(x^2-x,xy,0)$ with $C(F_{15})=\{2x^2-x=0\}$ and $\dim O(F_{15})=13$.

\item[(16)] $F_{16}=(x^2,xy,0)$ with $C(F_{16})=\{x=0\}$ and $\dim O(F_{16})=12$.

\item[(17a)] $F_{17}=(xy,x,y)$ with $C(F_{17})=\emptyset$ and $\dim O(F_{17})=14$.

\item[(17b)] $F_{17'}=(x^2+y^2,x,y)$ with $C(F_{17'})=\emptyset$ and $\dim O(F_{17})=14$.

\item[(18)] $F_{18}=(x^2,x,y)$ with $C(F_{18})=\emptyset$ and $\dim O(F_{18})=13$.

\item[(19a)] $F_{19}=(xy,x+y,0)$ with $C(F_{19})=\{y-x=0\}$ and $\dim O(F_{19})=13$.

\item[(19b)] $F_{19'}=(x^2+y^2,x,0)$ with $C(F_{19'})=\{y=0\}$ and $\dim O(F_{19'})=13$.

\item[(20)] $F_{20}=(x,xy,0)$ with $C(F_{20})=\{x=0\}$ and $\dim O(F_{20})=12$.

\item[(21)] $F_{21}=(x^2,y,0)$ with $C(F_{21})=\{2x=0\}$ and $\dim O(F_{21})=12$.

\item[(22)] $F_{22}=(x^2+y,x,0)$ with $C(F_{22})=\emptyset$ and $\dim O(F_{22})=11$.

\item[(23)] $F_{23}=(x^2,x,0)$ with $C(F_{23})=\R^2$ and $\dim O(F_{23})=10$.

\item[(24)] $F_{24}=(x,y,0)$ with $C(F_{24})=\emptyset$ and $\dim O(F_{24})=9$.

\item[(25a)] $F_{25}=(xy,0,0)$ with $C(F_{25})=\R^2$ and $\dim O(F_{25})=10$.

\item[(25b)] $F_{25'}=(x^2+y^2,0,0)$ with $C(F_{25'})=\R^2$ and $\dim O(F_{25'})=10$.

\item[(26)] $F_{26}=(x^2+y,0,0)$ with $C(F_{26})=\R^2$ and $\dim O(F_{26})=9$.

\item[(27)] $F_{27}=(x^2,0,0)$ with $C(F_{27})=\R^2$ and $\dim O(F_{27})=8$.

\item[(28)] $F_{28}=(x,0,0)$ with $C(F_{28})=\R^2$ and $\dim O(F_{28})=7$.

\item[(29)] $F_{29}=(0,0,0)$ with $C(F_{29})=\R^2$ and $\dim O(F_{29})=3$.

\end{itemize}

For $\Omega_{\R^2}(2,2,2,2)$ in addition to the mappings coming from $\Omega_{\R^2}(2,2,2)$ we obtain the following cases:

\begin{itemize}

\item[(1)] (generic case) $G_1=(x^2+y,y^2,xy,x)$ which is an immersion. $\dim O(G_1)=24$.

\item[(2)] $G_2=(x^2,y^2,xy,x)$ with singular point $(0,0)$, $SI_{G_2}=\{x=0\}$ and $\dim O(G_2)=23$.

\item[(3a)] $G_3=(x^2,y^2,x,y)$ which is an immersion. $\dim O(G_3)=22$

\item[(3b)] $G_{3'}=(x^2-y^2,xy,x,y)$ which is an immersion. $\dim O(G_{3'})=22$

\item[(4)] $G_4=(x^2,xy,x,y)$ which is an immersion. Moreover, $\dim O(G_4)=21$.

\end{itemize}

Similarly as over $\C$ for the quadratic mappings $\R^2\to\R^n$ for $n\geq 5$ the generic case is $G_0=(x^2,xy,y^2,x,y,0,\ldots,0)$ and all other cases can be obtained from the $\Omega_{\R^2}(2,2,2,2)$.

Obviously if $i:\C^k\to\C^n$ is the standard inclusion then $F:\C^2\to\C^k$ and $i\circ F:\C^2\to\C^n$ share their geometric properties. However the relation between $O(F)$ and $O(i\circ F)$ may not be completely obvious. It is explained in Lemma \ref{lemstab} and Corollary \ref{codimo}. For the convenience of the reader we will write down the dimensions of all the orbits of quadratic mappings $\C^2\to\C^n$, we have:

\begin{enumerate}
\item $\dim_a(G_0)=5$, so $\dim O(i\circ G_0)=6n$
\item $\dim_a(G_1)=4$, so $\dim O(i\circ G_1)=5n+4$
\item $\dim_a(G_2)=4$, so $\dim O(i\circ G_1)=5n+3$
\item $\dim_a(G_3)=4$, so $\dim O(i\circ G_1)=5n+2$
\item $\dim_a(G_4)=4$, so $\dim O(i\circ G_1)=5n+1$
\item $\dim_a(F_1)=3$, so $\dim O(i\circ F_1)=4n+6$
\item $\dim_a(F_2)=3$, so $\dim O(i\circ F_2)=4n+5$
\item $\dim_a(F_3)=3$, so $\dim O(i\circ F_3)=4n+4$
\item $\dim_a(F_4)=3$, so $\dim O(i\circ F_4)=4n+2$
\item $\dim_a(F_5)=3$, so $\dim O(i\circ F_5)=4n+5$
\item $\dim_a(F_6)=3$, so $\dim O(i\circ F_6)=4n+4$
\item $\dim_a(F_7)=2$, so $\dim O(i\circ F_7)=3n+6$
\item $\dim_a(F_8)=3$, so $\dim O(i\circ F_8)=4n+4$
\item $\dim_a(F_9)=3$, so $\dim O(i\circ F_9)=4n+3$
\item $\dim_a(F_{10})=2$, so $\dim O(i\circ F_{10})=3n+5$
\item $\dim_a(F_{11})=3$, so $\dim O(i\circ F_{11})=4n+3$
\item $\dim_a(F_{12})=2$, so $\dim O(i\circ F_{12})=3n+5$
\item $\dim_a(F_{13})=2$, so $\dim O(i\circ F_{13})=3n+4$
\item $\dim_a(F_{14})=3$, so $\dim O(i\circ F_{14})=4n+2$
\item $\dim_a(F_{15})=2$, so $\dim O(i\circ F_{15})=3n+4$
\item $\dim_a(F_{16})=2$, so $\dim O(i\circ F_{16})=3n+3$
\item $\dim_a(F_{17})=3$, so $\dim O(i\circ F_{17})=4n+2$
\item $\dim_a(F_{18})=3$, so $\dim O(i\circ F_{18})=4n+1$
\item $\dim_a(F_{19})=2$, so $\dim O(i\circ F_{19})=3n+4$
\item $\dim_a(F_{20})=2$, so $\dim O(i\circ F_{20})=3n+3$
\item $\dim_a(F_{21})=2$, so $\dim O(i\circ F_{21})=3n+3$
\item $\dim_a(F_{22})=2$, so $\dim O(i\circ F_{22})=3n+2$
\item $\dim_a(F_{23})=2$, so $\dim O(i\circ F_{23})=3n+1$
\item $\dim_a(F_{24})=2$, so $\dim O(i\circ F_{24})=3n$
\item $\dim_a(F_{25})=1$, so $\dim O(i\circ F_{24})=2n+4$
\item $\dim_a(F_{26})=1$, so $\dim O(i\circ F_{24})=2n+3$
\item $\dim_a(F_{27})=1$, so $\dim O(i\circ F_{24})=2n+2$
\item $\dim_a(F_{28})=1$, so $\dim O(i\circ F_{24})=2n+1$
\item $\dim_a(F_{29})=0$, so $\dim O(i\circ F_{24})=n$
\end{enumerate}

To better visualize the structure of the space of quadratic mappings we present Figure~\ref{Picture1}.

Each row consists of orbits of given dimension, from the largest to the smallest, and a rising path joins two orbits if the smaller is contained in the closure of the larger. To compare orbits of a mapping $\C^2\rightarrow\C^{n_1}$ and a mapping $\C^2\rightarrow\C^{n_2}$ for $n_1<n_2$ we compose the first one with the standard inclusion $\C^{n_1}\rightarrow\C^{n_2}$. In many cases the existence or the lack of an edge follows trivially from the characterization of the respective orbits. However in some cases it is not obvious whether containment holds or not, hence we provide a brief argument in Section \ref{secOCC}.

\newpage
\begin{figure}[H]
                \centering
\begin{tikzpicture}
\clip(-0.16,-2.1) rectangle (7.6,11.8);
\draw (0.5,11.5)-- (0.5,5.);

\draw (0.5,10.5)-- (2.,10);
\draw (2.,10)-- (2.,1.);
\draw (3.5,9.5)-- (3.5,-0.5);
\draw (5.,6.)-- (5.,3.);
\draw (6.5,5.)-- (6.5,0.);

\draw (0.5,9.)-- (2.,8.);
\draw (0.5,8.)-- (2.,7.);
\draw (0.5,8.)-- (3.5,7.);
\draw (0.5,7.)-- (3.5,6.);
\draw (0.5,7.)-- (5.,6.);
\draw (0.5,7.)-- (2.,5.);
\draw (0.5,5.)-- (2.,4.);
\draw (0.5,5.)-- (6.5,4.);
\draw (0.5,5.)-- (3.5,1.);

\draw (2.,10)-- (3.5,9.5);
\draw (2.,8.)-- (3.5,7.);
\draw (2.,7.)-- (3.5,6.);
\draw (2.,7.)-- (5.,6.);
\draw (2.,6.)-- (6.5,5.);
\draw (2.,5.)-- (5.,4.);
\draw (2.,4.)-- (6.5,3.);
\draw (2.,1.)-- (3.5,-0.5);

\draw (3.5,6.)-- (2.,5.);
\draw (3.5,6.)-- (5.,5.);
\draw (3.5,5.)-- (2.,4.);
\draw (3.5,4.)-- (6.5,3.);
\draw (3.5,3.)-- (2.,1.);
\draw (3.5,-0.5)-- (5.,-1.);

\draw (5.,6.)-- (6.5,5.);
\draw (5.,5.)-- (3.5,4.);
\draw (5.,4.)-- (3.5,3.);
\draw (5.,3.)-- (3.5,1.);
\draw (5.,3.)-- (2.,1.);
\draw (5.,-1.)-- (5.,-1.5);

\draw (6.5,5.)-- (5.0,4.);
\draw (6.5,5.)-- (3.5,4.);
\draw (6.5,4.)-- (5.0,3.);
\draw (6.5,2.)-- (2.,1.);
\draw (6.5,2.)-- (3.5,0.);
\draw (6.5,0.)-- (5.,-1.);

\begin{scriptsize}

\draw [fill=black] (0.5,11.5) circle (2.5pt);
\draw (0.9,11.5) node {$G_0$};
\draw [fill=black] (0.5,11) circle (2.5pt);
\draw (0.9,11) node {$G_1$};
\draw [fill=black] (0.5,10.5) circle (2.5pt);
\draw (0.9,10.5) node {$G_2$};
\draw [fill=black] (0.5,9.) circle (2.5pt);
\draw (0.9,9.) node {$F_1$};
\draw [fill=black] (0.5,8.) circle (2.5pt);
\draw (0.9,8.) node {$F_2$};
\draw [fill=black] (0.5,7.) circle (2.5pt);
\draw (0.1,7.) node {$F_3$};
\draw [fill=black] (0.5,5.) circle (2.5pt);
\draw (0.1,5.) node {$F_4$};

\draw [fill=black] (2.,10) circle (2.5pt);
\draw (2.4,10) node {$G_3$};
\draw [fill=black] (2.,8.) circle (2.5pt);
\draw (2.4,8.) node {$F_5$};
\draw [fill=black] (2.,7.) circle (2.5pt);
\draw (2.4,7.) node {$F_6$};
\draw [fill=black] (2.,6.) circle (2.5pt);
\draw (1.6,6.) node {$F_{11}$};
\draw [fill=black] (2.,5.) circle (2.5pt);
\draw (1.6,5.) node {$F_{14}$};
\draw [fill=black] (2.,4.) circle (2.5pt);
\draw (1.6,4.) node {$F_{18}$};
\draw [fill=black] (2.,1.) circle (2.5pt);
\draw (1.6,1.) node {$F_{23}$};

\draw [fill=black] (3.5,9.5) circle (2.5pt);
\draw (3.9,9.5) node {$G_4$};
\draw [fill=black] (3.5,7.) circle (2.5pt);
\draw (3.9,7.) node {$F_8$};
\draw [fill=black] (3.5,6.) circle (2.5pt);
\draw (3.9,6.) node {$F_9$};
\draw [fill=black] (3.5,5.) circle (2.5pt);
\draw (3.9,5.) node {$F_{17}$};
\draw [fill=black] (3.5,4.) circle (2.5pt);
\draw (3.1,4.) node {$F_{19}$};
\draw [fill=black] (3.5,3.) circle (2.5pt);
\draw (3.9,3.) node {$F_{20}$};
\draw [fill=black] (3.5,1.) circle (2.5pt);
\draw (3.9,1.) node {$F_{25}$};
\draw [fill=black] (3.5,0.) circle (2.5pt);
\draw (3.9,0.) node {$F_{26}$};
\draw [fill=black] (3.5,-0.5) circle (2.5pt);
\draw (3.1,-0.5) node {$F_{27}$};

\draw [fill=black] (5.,6.) circle (2.5pt);
\draw (5.4,6.) node {$F_7$};
\draw [fill=black] (5.,5.) circle (2.5pt);
\draw (5.4,5.) node {$F_{10}$};
\draw [fill=black] (5.,4.) circle (2.5pt);
\draw (5.4,4.) node {$F_{15}$};
\draw [fill=black] (5.,3.) circle (2.5pt);
\draw (5.4,3.) node {$F_{16}$};
\draw [fill=black] (5.,-1.) circle (2.5pt);
\draw (5.4,-1.) node {$F_{28}$};
\draw [fill=black] (5.,-1.5) circle (2.5pt);
\draw (5.4,-1.5) node {$F_{29}$};

\draw [fill=black] (6.5,5.) circle (2.5pt);
\draw (6.9,5.) node {$F_{12}$};
\draw [fill=black] (6.5,4.) circle (2.5pt);
\draw (6.9,4.) node {$F_{13}$};
\draw [fill=black] (6.5,3.) circle (2.5pt);
\draw (6.9,3.) node {$F_{21}$};
\draw [fill=black] (6.5,2.) circle (2.5pt);
\draw (6.9,2.) node {$F_{22}$};
\draw [fill=black] (6.5,0.) circle (2.5pt);
\draw (6.9,0.) node {$F_{24}$};

\end{scriptsize}
\end{tikzpicture}

                \caption{The orbits\label{Picture1}}
                \end{figure}
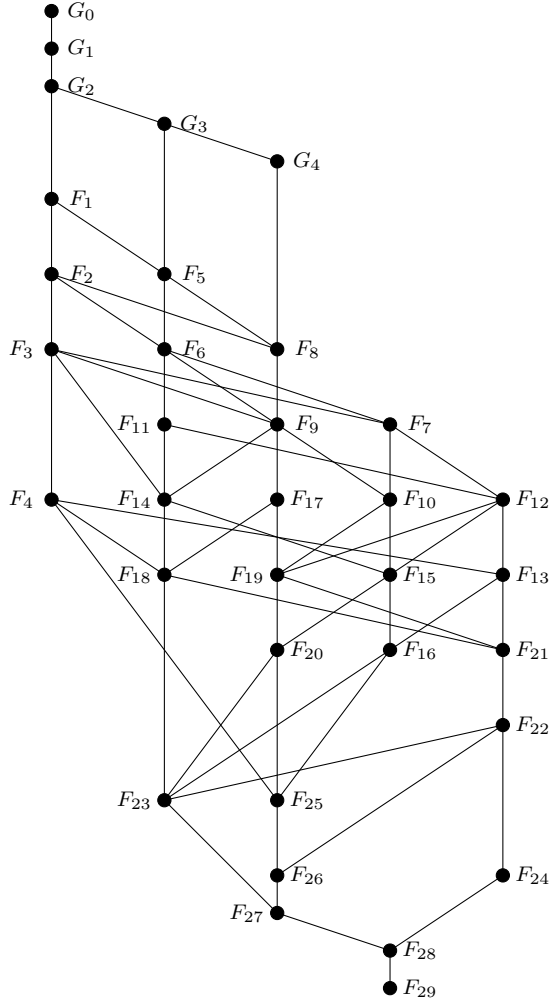

The paper is organized as follows:

In Section \ref{secMR} we introduce the notation and tools used throughout the paper. Afterwards we focus on the space $\Omega_{\C^2}(2,2,2)$. In Section \ref{secreal} we use the list obtained for complex quadratic mappings in $\Omega_{\C^2}(2,2,2)$ to produce the list of real quadratic mappings in $\Omega_{\R^2}(2,2,2)$. In Section \ref{secHD} we describe the spaces of quadratic mappings $\C^2\to\C^n$ and $\R^2\to\R^n$ for $n\geq 4$. In Section \ref{secOCC} we deal with the containment of closures of orbits. Finally in Section \ref{top} we describe all topological types.

\section{Main Result}\label{secMR}

We start by making the following useful definition:

\begin{defi}
Let $F=(f_1,\ldots,f_m):\K^n\rightarrow\K^m$ be a polynomial mapping. We denote by $\dim_a(F)$ the dimension of the affine space spanned by the image of $F$. Moreover we denote by $\dim_q(F)$ the dimension of the linear space spanned by the image of $F^{(2)}$, where $F^{(2)}=(f_1^{(2)},\ldots,f_m^{(2)})$ and $f_i^{(2)}$ is the homogeneous part of $f_i$ of degree $2$. 
\end{defi}

Obviously for a quadratic mapping $F:\K^2\rightarrow\K^m$ we have $\dim_a(F)\leq 5$ and $\dim_q(F)\leq 3$. 

Let us also recall the following:

\begin{defi}
By $GA(n)$ we denote the group of affine transformations of $\K^n$. By $\mathcal{GA}(n,m)$ we denote the group $GA(m) \times GA(n)$ with the product given by formula: $(L_2,R_2)\circ (L_1,R_1)=(L_2\circ L_1, R_1\circ R_2)$. The group $\mathcal{GA}(n,m)$ acts on the set $\Omega_{\K^n}(2,\ldots,2)$ of quadratic polynomial mappings $\K^n\to\K^m$ as follows: $(L,R)F=L\circ F \circ R$. We denote the orbit of $F$ by $O(F)$. We say that $F$ and $G$ are linearly equivalent if there is $\alpha\in \mathcal{GA}(n,m)$ such that $F=\alpha G$, i.e. $F\in O(G)$.
\end{defi}

We denote by $\Stab (F)$ the stabilizer of $F:\K^n\to\K^m$. Note that $\dim \Stab(F)+\dim O(F)=\dim \mathcal{GA}(n,m)=n(n+1)+m(m+1)$. Moreover if $(L,R)\in \Stab(F)$ and the image of $F$ is not contained in an affine subspace of $\K^m$ then $R$ uniquely determines $L$. More generally, we have the following:

\begin{lem}\label{lemstab}
Let $F:\K^n\rightarrow\K^m$ be a polynomial mapping and let $\mathcal{G}$ be the subgroup of $GA(m)$ consisting of transformations which are identity after restricting to the affine space spanned by $F(\K^n)$. If $(L_0,R_0)\in\Stab(F)$ then the set $\{L\in GA(m):\ (L,R_0)\in\Stab(F)\}$ is equal to $L_0\mathcal{G}$, in particular it has the structure of an $m(m-\dim_a(F))$--dimensional affine variety.
\end{lem}
\begin{proof}
First observe that $\mathcal{G}=\{L\in GA(m):\ (L,\id_{\K^n})\in\Stab(F)\}$. Obviously for $L\in \mathcal{G}$ we have $L_0\circ L\circ F\circ R_0=L_0\circ F\circ R_0=F$, so $(L_0\circ L,R_0)\in\Stab(F)$. On the other hand, if $(L,R_0)\in\Stab(F)$ then $L\circ F\circ R_0=L_0\circ F\circ R_0$, so $L_0^{-1}\circ L\circ F=F$, thus $L_0^{-1}\circ L\in\mathcal{G}$.
\end{proof}

As a direct consequence of the Lemma we obtain the following:

\begin{co}\label{codimo}
Let $F:\K^n\rightarrow\K^m$ be a polynomial mapping and let $G=(F,0,\ldots,0):\K^n\rightarrow\K^s$ be the composition of the standard inclusion $\K^m\rightarrow\K^s$ and $F$. We have $\dim\Stab(G)=\dim\Stab(F)+s^2-m^2-(s-m)\dim_a(F)$ and $\dim O(G)=\dim O(F)+(s-m)(1+\dim_a(F))$.
\end{co}

We will also use the following:

\begin{defi}\label{defCF}
Let $F\in\Omega_{\K^2}(2,2,2)$. We denote by $\mu(F)$ the topological degree of $F$, i.e. the number of points in $F^{-1}(F(x))$ for a generic $x\in\K^2$. For $\mu(F)=1$ we define the \emph{self-intersection curve} as the set 

$$SI_F=\{x\in\K^2:\ F^{-1}(F(x)) \text{ is not a simple point}\}$$ 
\end{defi}

Obviously in Definition \ref{defCF} we take $F^{-1}(F(x))$ with a scheme structure. We can also introduce a scheme structure on $SI_F$ in the following way, which we also use to actually compute $SI_F$: let $F=(f_1,f_2,f_3)$ and let $\pi:\K[x_1,y_1]\rightarrow\K[x_1,y_1,x_2,y_2]$ be the inclusion corresponding to the projection on the first coordinate in $\K^2\times\K^2$. Let $I$ be the ideal generated by $f_i(x_1,y_1)-f_i(x_2,y_2)$ for $i=1,2,3$. We define $SI_F$ as the scheme corresponding to the ideal $\pi^{-1}(I:(x_1-x_2,y_1-y_2))$.

For $F=(f_1,f_2,f_3)\in\Omega_{\K^2}(2,2,2)$ we define:
$$
J_1(F)=
\left| \begin{array}{cc}
\pa{f_2}{x} & \pa{f_2}{y}\\
\pa{f_3}{x} & \pa{f_3}{y}
\end{array} \right|,
\quad
J_2(F)=
\left| \begin{array}{cc}
\pa{f_3}{x} & \pa{f_3}{y}\\
\pa{f_1}{x} & \pa{f_1}{y}
\end{array} \right|,
\quad
J_3(F)=
\left| \begin{array}{cc}
\pa{f_1}{x} & \pa{f_1}{y}\\
\pa{f_2}{x} & \pa{f_2}{y}
\end{array} \right|.
$$ 

\begin{defi}
We call the set $C(F)=\{J_1(F)=J_2(F)=J_3(F)=0\}$ the \emph{critical set} of $F$, however when considering multiplicities of points we refer rather to the scheme $V(J_1(F),J_2(F),J_3(F))$. We also use the \emph{critical space} $Cs(F)=\langle J_1(F),J_2(F),J_3(F)\rangle$, which is the vector space generated over $\K$ by $J_i(F)$ seen as vectors in the six dimensional space of quadratic polynomials.
\end{defi}

Let $L\in GA(3)$, let $P\in GL(3)$ be the linear part of $L$ and let $C$ be the cofactor matrix of $P$, i.e. $C=(\det(P)P^{-1})^T$. It is easy to check that we have:

$$\left[\begin{matrix}J_1(L\circ F)\\J_2(L\circ F)\\J_3(L\circ F)\end{matrix}\right]
=C\left[\begin{matrix}J_1(F)\\J_2(F)\\J_3(F)\end{matrix}\right]$$

In particular we obtain that the spaces $Cs(F)$ and $Cs(L\circ F)$ are equal, so composition with $L$ preserves the critical space of $F$. Consequently if $(L,R)\in \Stab(F)$ then $R$ also preserves $Cs(F)$.

From now on to the end of this section we will focus on the space $\Omega_{\C^2}(2,2,2)$.

Let us now prove the following:

\begin{theorem}\label{thf1}
Let $F\in \Omega_{\C^2}(2,2,2)$. The following conditions are equivalent:
\begin{itemize}
\item[(1)] $F$ has three singular points.

\item[(2)] $O(F)$ is an open dense subset of $\Omega_{\C^2}(2,2,2)$.

\item[(3)] $O(F)=O(F_1)$, i.e. $F$ is linearly equivalent to $F_1=(x^2+y,y^2+x,xy)$.
\end{itemize}
Moreover, for such $F$ the set $SI_F$ is a union of three lines in a general position and $\dim O(F)=18$.

\end{theorem}

\begin{proof}
$(1)\Rightarrow (2)$ We will show that $\Stab (F)$ is finite. If $(L,R)\in \Stab(F)$ then $R$ must preserve the critical set. Note that the singular points are not co-linear. Indeed, they are the zero set of polynomials $J_i(F)$. Note that $V(J_i(F))$ is a smooth conic or two lines or one line. If the singular points are co-linear then the first possibility is excluded. Hence $V(J_i(F))$ is one or two lines and the three singular points must lie on exactly one line included in $V(J_i(F))$. This means that all three $V(J_i(F))$ contain a common line, so the set of critical points of $F$ contains a line - a contradiction.

Now it is easy to see that there is only a finite number (at most six) of such linear mappings $R$. Moreover, the set $F(\C^2)$ is not contained in a hyperplane because if it were then all three $J_i(F)$ would be proportional. This implies that $L$ is uniquely determined by $R$ and $\#\Stab (F)\le 6$. Since $\Stab (F)$ is finite we know that $\dim O(F)=\dim \Omega_{\C^2}(2,2,2)$ and $O(F)$ is a maximal orbit, in particular it is open and dense.

$(2)\Rightarrow (3)$ Note that the mapping $F_1$ has three singular points: $(\frac{\varepsilon}{2},\frac{\varepsilon^2}{2})$ for $\varepsilon^3=1$. Hence $O(F_1)$ is open and dense, in particular $O(F)\cap O(F_1)\neq \emptyset$.

$(3)\Rightarrow (1)$  Obvious since $F_1$ has three singular points.

To prove the last remark observe that $SI_{F_1}$ is a union of three lines: $x+\varepsilon^2 y-\varepsilon=0$ for $\varepsilon^3=1$.

\end{proof}

The mapping $F_1$ has a quite nice geometry which is shared with mappings from $O(F_1)$. The curve $SI_{F_1}$ consists of three lines $x+\varepsilon^2 y-\varepsilon=0$ for $\varepsilon^3=1$ forming a triangle, the singular points $(\frac{\varepsilon}{2},\frac{\varepsilon^2}{2})$ are midpoints of the triangles edges. After restricting $F_1$ to one of the edges we obtain the mapping 
$$(x+\frac{\varepsilon}{2},-\varepsilon x+\frac{\varepsilon^2}{2})\mapsto(x^2+\frac{3\varepsilon^2}{4},\varepsilon^2x^2+\frac{3\varepsilon}{4},-\varepsilon x^2+\frac{1}{4}).$$
Obviously this restriction is equivalent to the endomorphism $x\mapsto x^2$ of a line. The vertices of the triangle $(-\varepsilon,-\varepsilon^2)$ are the preimage of the unique triple point of $F_1$.

Note that it follows from the complete classification that also the following conditions are equivalent:
\begin{itemize}
\item $F$ has three singular points,
\item $SI_F$ is a union of three lines,
\item $\mu(F)=1$ and $F$ has a (unique) triple point.
\end{itemize}

Before we state the next theorem we will do some preparatory work.

Let $F=(f_1,f_2,f_3)\in\Omega_{\C^2}(2,2,2)$ and let $f_i=a_ix^2+b_ixy+c_iy^2+d_ix+e_iy+g_i$ for $i=1,2,3$. Moreover, let
$$\Phi_1(F)=\left[\begin{matrix}a_1&c_1&b_1\\a_2&c_2&b_2\\a_3&c_3&b_3\end{matrix}\right].$$
First we will consider the case when $\det \Phi_1(F)\neq 0$, i.e. $\dim_q(F)=3$. Geometrically it means that $O(F)\cap\Omega_{\C^2}(2,2,1)=\emptyset$. We will show that in this case either $(x^2,y^2,xy)\in O(F)$ or $(x^2+y,y^2+x,xy+ax+by)\in O(F)$ for some $a,b\in\C$. Observe that if $L$ is the mapping associated with $\Phi_1^{-1}(F)$ then $L\circ F=(x^2+d_1'x+e_1'y+g_1',y^2+d_2'x+e_2'y+g_2',xy+d_3'x+e_3'y+g_3')$, where $(d_1',d_2',d_3')=\Phi_1^{-1}(F)(d_1,d_2,d_3)$ and $(e_1',e_2',e_3')=\Phi_1^{-1}(F)(e_1,e_2,e_3)$. Moreover, if $R(x,y)=(x-\frac{d_1'}{2},y-\frac{e_2'}{2})$ and $L'(x,y,z)=(x-g_1'+(\frac{d_1'}{2})^2+\frac{d_1'e_2'}{2},y-g_2'+(\frac{e_2'}{2})^2+\frac{d_1'e_2'}{2},z+\frac{-d_1'e_2'+d_1'd_3'+e_2'e_3'}{2}-g_3')$ then $F'=L'\circ L\circ F\circ R=(x^2+e_1''y,y^2+d_2''x,xy+d_3''x+e_3''y)$.

Let $\Omega_1=\{F\in\Omega_{\C^2}(2,2,2):\ F=(x^2+e_1''y,y^2+d_2''x,xy+d_3''x+e_3''y)\}$. We constructed above a mapping $\Theta:\Omega_{\C^2}(2,2,2)\setminus V(\det \Phi_1)\ni F\mapsto F'\in\Omega_1$ which is compatible with the action of $\mathcal{GA}(2,3)$ in the sense that $\Theta^{-1}(F')\subset O(F')$. In fact we constructed an isomorphism $\Omega_{\C^2}(2,2,2)\setminus V(\det \Phi_1)\cong\Omega_1\times GL(3)\times\C^5$ and the mapping $\Theta$ may be viewed as projection.

Next, for $F\in\Omega_1\setminus V(e_1''d_2'')$ we take $R'(x,y)=(({e_1''}^2d_2'')^{\frac{1}{3}}x,(e_1''{d_2''}^2)^{\frac{1}{3}}y)$ and $L''(x,y,z)=(({e_1''}^2d_2'')^{\frac{-2}{3}}x,(e_1''{d_2''}^2)^{\frac{-2}{3}}y,(e_1''d_2'')^{-1}z)$ and obtain $L''\circ F\circ R'=(x^2+y,y^2+x,xy+ax+by)$, where $a=d_3''(e_1''{d_2''}^2)^{\frac{-1}{3}}$ and $b=e_3''({e_1''}^2d_2'')^{\frac{-1}{3}}$. In particular we obtain a mapping $\Theta_1:\Omega_1\setminus V(e_1''d_2'')\rightarrow\Omega_2$ compatible with the action of $\mathcal{GA}(2,3)$, where $\Omega_2=\{F\in\Omega_{\C^2}(2,2,2):\ F=(x^2+y,y^2+x,xy+ax+by)\}$.

Finally, consider the mapping $\Theta_{\alpha,\beta}:(F(x,y))\mapsto F(x+\alpha,y+\beta)$ with $\alpha\beta\neq 1$. One can calculate that for $F\in\Omega_1$ we have $\Theta\circ\Theta_{\alpha,\beta}(F)=
(x^2+(e_1''+\alpha^3d_2''-2\alpha^2d_3''-2\alpha e_3'')/(1-\alpha\beta)^2 y,
y^2+(\beta^3e_1''+d_2''-2\beta d_3''-2\beta^2e_3'')/(1-\alpha\beta)^2 x,xy+\ldots)$. In particular for every $F\in\Omega_1\setminus V(e_1'',d_2'',d_3'',e_3'')$, i.e. every $F\in\Omega_1$ except $(x^2,y^2,xy)$, there are $\alpha,\beta$ such that $\Theta\circ\Theta_{\alpha,\beta}(F)\in\Omega_1\setminus V(e_1''d_2'')$. Consequently we may construct a mapping $\Omega_1\setminus V(e_1'',d_2'',d_3'',e_3'')\rightarrow \Omega_2$ compatible with the action of $\mathcal{GA}(2,3)$. In particular we obtain:

$$\Omega_{\C^2}(2,2,2)=\mathcal{GA}(2,3)\Omega_{\C^2}(2,2,1)\cup\mathcal{GA}(2,3)\Omega_2\cup O(x^2,y^2,xy)$$

Now we can prove the following:

\begin{theorem}\label{thf2}
Let $F\in \Omega_{\C^2}(2,2,2)$. The following conditions are equivalent:
\begin{itemize}

\item[(1)] $F$ has two singular points.

\item[(2)] $F'=(x^2+y, y^2+x, xy+ax+by)\in O(F)$ for some $a,b\in\C$ such that $2a^3 + a^2b^2 + \frac{9}{2} a b + 2 b^3 - \frac{27}{16}=0$ and  $(a,b)\not=(-\frac{3}{2}\epsilon, -\frac{3}{2}\epsilon^2)$, where $\epsilon^3=1.$

\item[(3)] $O(F)=O(F_2)$, i.e. $F$ is linearly equivalent to $F_2=(x^2+y,y^2+x,xy+\frac{1}{2}x+\frac{1}{2}y)$.
\end{itemize}
Moreover, for such $F$ the set $SI_F$ is a union of two transversal lines and $\dim O(F)=17$.
\end{theorem}
\begin{proof}
$(1)\Rightarrow (2)$ First we show that $\dim_q(F)=3$. Suppose that $\dim_q(F)=2$, then we can find $F''=(f_1,f_2,f_3)\in O(F)\cap\Omega_{\C^2}(2,2,1)$. Observe that $J_1(F)$ and $J_2(F)$ have degree at most $1$. Since $F''$ has two singular points both $V(J_1(F))$ and $V(J_2(F))$ must contain the line through those points, i.e. be this line or $\C^2$. It follows that either $f_3$ is constant or $V(J_3(F))$ also contains a line through the two points. In both cases $C(F)$ is not finite - a contradiction.

Now observe that $F\notin O(x^2,y^2,xy)$ since $C(x^2,y^2,xy)=\{(0,0)\}$. Hence we have $F\in\mathcal{GA}(2,3)\Omega_2$, in particular $F'=(x^2+y, y^2+x, xy+ax+by)\in O(F)$ for some $a,b\in\C$.

We have $J_3(F')=4xy-1$, $-J_2(F')=2x^2+2xb-y-a$, $J_1(F')=2y^2+2ya-x-b$. Hence $C(F')$ is given by the equations $x^3 + x^2b - \frac{1}{2}xa - \frac{1}{8}=0$, $y=\frac{1}{4x}$.

Note that the mapping $F'$ has at most two singular points precisely for those $(a,b)$ for which the discriminant of polynomials
$x^3 + x^2b - \frac{1}{2}xa - \frac{1}{8}$ and $3x^2+2xb-\frac{1}{2}a$ vanishes.
The set of such mappings in $\Omega_2$ is given by the equation
$2a^3 + a^2b^2 + \frac{9}{2}ab + 2b^3 - \frac{27}{16}=0$. Moreover, only for $(a,b)=(-\frac{3}{2}\epsilon, -\frac{3}{2}\epsilon^2)$ the mapping $F'$ has exactly one singular point.

$(2)\Rightarrow (3)$ Let $\frac{\alpha}{2}$ be a double root of the equation $x^3 + x^2b - \frac{1}{2}xa - \frac{1}{8}$, obviously $\frac{1}{2\alpha^2}$ is the other root. We have $a=-(\alpha^3+2)/2\alpha$ and $b=-(2\alpha^3+1)/2\alpha^2$. The singular points of $F'$ are $(\frac{\alpha}{2},\frac{1}{2\alpha})$ with multiplicity $2$ and $(\frac{1}{2\alpha^2},\frac{\alpha^2}{2})$. By direct computation we see that the self-intersection curve of the mapping $(x^2+y, y^2+x, xy+ax+by)$ is given by the equation
$$x^3 -2x^2ya -2x^2a^2 -x^2b -2 xy^2b -4xyab +3xy-2xa^2b +xa -xb^2 $$
$$+y^3 -y^2a -2y^2b^2 -ya^2 -2 yab^2 +yb +a^3 + 3ab +b^3 +1=0.$$

This equation can be also expressed as $$(\alpha x+y-\frac{\alpha^3+1}{2\alpha})^2(\frac{1}{\alpha^2}x+y-\frac{\alpha^6+1}{2\alpha^4})=0.$$ Observe, that the double line passes through both singular points and the simple line passes through the simple singular point. This is expected since we have a degeneration of the general case in which two edges of a triangle are brought together. Now if $(L,R)\in \Stab(F')$ then $R$ transforms $SI_{F'}$ onto $SI_{F'}$ and singular points into singular points and preserves their multiplicity. Hence if we choose a system of coordinates such that $SI_{F'}$ is given by $x_1^2y_1=0$ then $R(x_1,y_1)=(tx_1,y_1)$, $t\in\C^*$. Thus $\dim\Stab(F')\le 1$ and since $F'$ is not generic $\dim\Stab(F)=\dim\Stab(F')=1$. Hence $\dim O(F)=17$.

Now observe that since $F'\in O(F)\cap \Omega_2$ the intersection $O(F)\cap \Omega_2$ is a curve. Moreover the intersection is contained in the irreducible curve $V_0=V(2a^3 + a^2b^2 + \frac{9}{2}ab + 2b^3 - \frac{27}{16})$. So $O(F)\cap \Omega_2$ is open and dense in $V_0$. Similarly $O(F_2)\cap \Omega_2$ is open and dense in $V_0$, so $O(F)=O(F_2)$.

$(3)\Rightarrow (1)$ Obvious.

\end{proof}

Similarly as after Theorem \ref{thf1} we may add an additional characterization of mappings in $O(F_2)$ which follows from the complete classification rather then from an explicit proof. We have that the following conditions are equivalent:
\begin{itemize}
\item $C(F)$ has two points (one ordinary and one double),
\item $SI_F$ consists of two lines (one ordinary and one double),
\item $\mu(F)=1$ and $F^{-1}(F(x))$ consists of a double point and an ordinary point for some (unique) $x\in\C^2$.
\end{itemize}

Furthermore we have the following:

\begin{theorem}\label{thf3}
Let $F\in \Omega_{\C^2}(2,2,2)$. The following conditions are equivalent:
\begin{itemize}
\item[(1)] $C(F)$ is a multiple (triple) point and $\mu(F)=1$.

\item[(2)] $F'=(x^2+y, y^2+x, xy+ax+by)\in O(F)$ for $a=-\frac{3}{2}\epsilon$ and $b=-\frac{3}{2}\epsilon^2$, where $\epsilon^3=1.$

\item[(3)] $O(F)=O(F_3)$, i.e. $F$ is linearly equivalent to $F_3=(x^2,y^2+x,xy)$.

\end{itemize}
Moreover, for such $F$ the set $SI_F$ is a line and $\dim O(F)=16$.
\end{theorem}

\begin{proof}
$(1)\Rightarrow (2)$ Similarly as in the proof of Theorem \ref{thf2} observe that $\dim_q(F)=3$ because otherwise $C(F)$ would be empty, an ordinary point or not finite. Moreover $F\notin O(x^2,y^2,xy)$ since $\mu(x^2,y^2,xy)=2$. Hence we have $F\in\mathcal{GA}(2,3)\Omega_2$, in particular $F'=(x^2+y, y^2+x, xy+ax+by)\in O(F)$ for some $a,b\in\C$. Finally, $F'$ has one critical point only for such $(a,b)$ for which the polynomials
$x^3 + x^2b - \frac{1}{2}xa - \frac{1}{8}$, $3x^2+2xb-\frac{1}{2}a$, $6x+2b$ have a common zero, i.e, only for  $(a,b)=(-\frac{3}{2}\epsilon, -\frac{3}{2}\epsilon^2)$.

$(2)\Rightarrow (3)$ We have $(x^2+y, y^2+x, xy-\frac{3}{2}\epsilon x-\frac{3}{2}\epsilon^2 y)=(\epsilon x, \epsilon^2y, z)\circ (x^2+y, y^2+x, xy-\frac{3}{2}x-\frac{3}{2}y)\circ (\epsilon x, \epsilon^2 y)$. Moreover for $L=(x+y+2z+1,y-\frac{3}{4},y+z+\frac{1}{2})$ and $R=(x-y+\frac{1}{2},y+\frac{1}{2})$ we have $F_3=L\circ(x^2+y, y^2+x, xy-\frac{3}{2}x-\frac{3}{2}y)\circ R$.

$(3)\Rightarrow (1)$ Obvious.

We have $SI_{F_3}=V(x^3)$ and $C(F_3)=V(x^2,xy,2y^2-x)$. Thus if $(L,R)\in \Stab(F_3)$ then $R$ transforms the line $\{x=0\}$ onto itself and preserves the critical point $(0,0)$, so $R=(\alpha x,\beta x+\gamma y), \alpha,\gamma \in\C^*, \beta \in \C$. However $R$ preserves not only the point $(0,0)$ but the scheme $V(x^2,xy,2y^2-x)$. We have $C(F_3\circ R)=V(x^2,xy,2\gamma^2y^2-\alpha x)$, thus $\alpha=\gamma^2$. So $\dim\Stab(F)=\dim\Stab(F_3)=2$ and $\dim O(F)=16$.
\end{proof}

\begin{theorem}\label{thf4}
Let $F\in \Omega_{\C^2}(2,2,2)$. The following conditions are equivalent:
\begin{itemize}
\item[(1)] $C(F)$ is a multiple (triple) point and $\mu(F)>1$.
\item[(2)] $O(F)=O(F_4)$, i.e. $F$ is linearly equivalent to $F_4=(x^2,y^2,xy)$.
\end{itemize}
Moreover, $\dim O(F)=14$.
\end{theorem}

\begin{proof}
Similarly as in the proof of Theorem \ref{thf2} observe that $\dim_q(F)=3$. Moreover, $\mu(F_1)=\mu(F_2)=\mu(F_3)<\mu(F)$, hence from Theorems \ref{thf1}, \ref{thf2} and \ref{thf3} we obtain $F\notin\mathcal{GA}(2,3)\Omega_2$. Recall that $\Omega_{\C^2}(2,2,2)=\mathcal{GA}(2,3)\Omega_{\C^2}(2,2,1)\cup\mathcal{GA}(2,3)\Omega_2\cup O(F_4)$, thus $F\in O(F_4)$.

Recall that we have an isomorphism $\Omega_{\C^2}(2,2,2)\setminus V(\det \Phi_1)\cong\Omega_1\times GL(3)\times\C^5$ and $\Theta:\Omega_{\C^2}(2,2,2)\setminus V(\det \Phi_1)\to\Omega_1$ may be viewed as projection. Since $O(F_4)=\Theta^{-1}(V(e_1'',d_2'',d_3'',e_3''))$
we have $\dim O(F_4)=14$.
\end{proof}

Theorem \ref{thf4} concludes the analysis of $\Omega_{\C^2}(2,2,2)\setminus V(\det(\Phi_1))$ so now we will shift our attention to the case when $\det(\Phi_1(F))=0$. As mentioned before, in this case we have $\dim_q(F)\leq 2$, i.e. $O(F)\cap\Omega_{\C^2}(2,2,1)\neq\emptyset$. In particular we may assume that $F=(f_1,f_2,ax+by)$, where $(f_1,f_2)\in\Omega_{\C^2}(2,2)$ is one of mappings listed in the paper \cite{fj} and $a,b\in\C$. We will start with the first pair from the list, namely $(f_1,f_2)=(x^2+y, y^2+x)$.

\begin{theorem}\label{thf5}
Let $F\in \Omega_{\C^2}(2,2,1)$ be of the form $(x^2+y, y^2+x, ax+by)$. The following conditions are equivalent:
\begin{itemize}
\item[(1)] $F$ has one singular point.
\item[(2)] $a,b \in \C^*$.
\item[(3)] $O(F)=O(F_5)$, i.e. $F$ is linearly equivalent to $F_5=(x^2,y^2,x+y)$.

\end{itemize}
Moreover, for such $F$ the set $SI_F$ is a line and $\dim O(F)=17$.
\end{theorem}  

\begin{proof}
$(1)\Rightarrow (2)$ From direct computation we obtain that $F$ has a singular point iff $4xy-1=0$, $2bx-a=0$, $b-2ay=0$ this equations hold only for $x=\frac{a}{2b}$, $y=\frac{b}{2a}$ so $a\neq 0$ and $b\neq 0$.

$(2)\Rightarrow (3)$ Consider the mappings $L=(a^2x-\frac{a^2}{b}z+\frac{a^4}{4b^2},b^2y-\frac{b^2}{a}z+\frac{2b^4-a^4}{4a^2},z-\frac{a^3+b^3}{4a^2})$ and $R=(\frac{x}{a}+\frac{a}{2b},\frac{y}{b}+\frac{b}{2a})$, we have $L\circ F\circ R=F_5$.

$(3)\Rightarrow(1)$ Obvious.

Since $F_5\notin O(F_1)$ we have $\dim O(F_5)\leq 17$. It remains to show that $\dim \Stab(F_5)\leq 1$, so let $(L,R)\in \Stab(F_5)$. We have $Cs(F)=\lra{x,y,xy}$ hence $R$ preserves the critical point $(0,0)$. Moreover, since $xy$ is the only quadratic term it also must be preserved up to scalar multiplication, i.e. $R$ preserves the pair $\{(1:0),(0:1)\}$ of points at infinity. Consequently $R=(\alpha x,\beta y)$ or $R=(\alpha y,\beta x)$, where $\alpha\beta\in\C^*$. However $R$ must also preserve $SI_{F_5}=\{(x,y): x+y=0\}$, hence $\alpha=\beta$. Since $R$ determines $L$ we have $\Stab(F_5)\cong \C^*\times\mathbb{Z}_2$. 
\end{proof}

\begin{theorem}\label{thf6}
Let $F=(x^2+y,y^2+x,ax+by)$. If either $a\in \C^*$, $b=0$ or $a=0$, $b\in \C^*$ then $F\in O(F_6)$, i.e. $F$ is linearly equivalent to $F_6=(x^2+y,y^2,x)$. Moreover, $\dim O(F_6)=16$.
\end{theorem}

\begin{proof}
If $a \neq 0$ and $b=0$ then we have $F=(x^2+y,y^2+x,ax)$. Moreover,  $(x,y-\frac{1}{a}z, \frac{1}{a}z)\circ F=F_6$. Second case is analogous. 

Obviously $O(F_6)\subset V(\det(\Phi_1))$, moreover $V(\det(\Phi_1))$ is irreducible and by Theorem \ref{thf5} it has the same dimension as its subset $O(F_5)$. Hence $\dim O(F_6)\leq 16$. It remains to show that $\dim \Stab(F_6)\leq 2$, so let $(L,R)\in \Stab(F_6)$. We have $Cs(F)=\lra{1,y,xy}$, hence $L$ and consequently $R$ preserve the terms $y$ and $xy$ up to scalar multiplication. Hence $R=(\alpha_1 x+\alpha_2,\beta_1y+\beta_2)$, where $\alpha_1\beta_1\in\C^*$ and $\alpha_2,\beta_2\in\C$. Moreover we have $Cs(F\circ R)=\lra{1,y,x(\beta_1y+\beta_2)}$, so $\beta_2=0$. Let $L'(x,y,z)=(\frac{x}{\alpha_1^2}-\frac{2\alpha_2z}{\alpha_1}+\frac{\alpha_2^2}{\alpha_1^2},\frac{y}{\beta_1^2},\frac{z-\alpha_2}{\alpha_1})$, we have $L'\circ F\circ R=(x^2+\frac{\beta_1}{\alpha_1^2}y,y^2,x)$. Obviously the $L''=L\circ(L')^{-1}$ such that $L''\circ (L'\circ F\circ R)=F_6$ exists if and only if $\beta_1=\alpha_1^2$. Which shows that indeed $\dim \Stab(F_6)\leq 2$.
\end{proof}

The last remaining case for $F=(x^2+y,y^2+x,ax+by)$ is $a=b=0$, i.e $F=F_7$. The mapping $(x^2+y,y^2+x)$ was thoroughly examined in \cite{fj} and the only substantial change we obtain while passing to $F_7$ lies in the size of the orbit and stabilizer (see Corollary \ref{codimo}). Indeed, if $(L,R)\in \Stab(F_7)$ then $R$ is one of the six permutations of the cusps of $(x^2+y,y^2+x)$. However, since the image of $F_7$ lies in a plane it does not determine $L$. Indeed we have $L\circ F_7=F_7$ for $L(x,y,z)=(x+\alpha_1z,y+\alpha_2z,\alpha_3z)$. Consequently $\dim O(F_7)=15$.

We now pass to the second mapping from the list in \cite{fj}, i.e. $(x^2+y,xy)$. A mapping in $O(x^2+y,xy)$ is characterized by the condition that its Jacobian is a parabola, in particular the quadratic part of the Jacobian must be a square. Let $\Phi_2(J)=4AC-B^2$ for $J=Ax^2+Bxy+Cy^2+\ldots$, then the quadratic part of $J$ is a square if and only if $\Phi_2(J)=0$. Hence a mapping $F\in O(x^2+y,xy,ax+by)$ must satisfy the conditions $\Phi_1(F)=0$ and $\Phi_2(J_1(F))=\Phi_2(J_2(F))=\Phi_2(J_3(F))=0$. Note that in general it is possible that $\Phi_2(J_i(F))=0$ for $i=1,2,3$ but there is a linear combination of $J_i(F)$ which does not have a square quadratic part. However the condition $\Phi_1(F)=0$ implies that the quadratic parts of $J_i(F)$ are proportional. Hence the variety $V(\Phi_1,\Phi_2(J_1),\Phi_2(J_2),\Phi_2(J_3))\subset\Omega_{\C^2}(2,2,2)$ is locally given by two independent equations, i.e. has dimension $16$.

\begin{theorem}\label{thf8}
Let $F\in \Omega_{\C^2}(2,2,2)$ be of the form $(x^2+y, xy, ax+by)$. The following conditions are equivalent:
\begin{itemize}
\item[(1)] $F$ has one singular point.

\item[(2)] $b\in\C^*$.

\item[(3)] $O(F)=O(F_8)$, i.e. $F$ is linearly equivalent to $F_8=(x^2,xy,y)$.

\end{itemize}
Moreover, for such $F$ the set $SI_F$ is a line and $\dim O(F_8)=16$.
\end{theorem}

\begin{proof}
$(1)\Rightarrow (2)$ From direct computation we obtain that $F$ has a singular point iff $2x^2-y=0$, $2bx-a=0$, $by-ax=0$. Those equations hold only for $x=\frac{a}{2b}$ so $b\neq 0$.

$(2)\Rightarrow (3)$ If $a=0$ then $L_0\circ F=F_8$ for $L_0=(x-\frac{z}{b},y,\frac{z}{b})$. So assume $a\neq 0$ and consider the mappings $L_1=(\frac{b^2}{a^2}x,\frac{b^3}{a^3}y,\frac{b}{a^2}z)$ and $R_1=(\frac{a}{b}x,\frac{a^2}{b^2}y)$, we have $F'=L_1\circ F\circ R_1=(x^2+y,xy,x+y)$.
Now consider $L=(x-z+\frac{1}{4},y+x-\frac{3}{2}z+\frac{1}{2},z-1)$ and $R_2=(x+\frac{1}{2},y-x-\frac{1}{2})$, we have $F_8=L_2\circ F'\circ R_2$.

$(3)\Rightarrow(1)$ Obvious.

Since $O(F_8)$ is contained in $V(\Phi_1,\Phi_2(J_1),\Phi_2(J_2),\Phi_2(J_3))$ it has dimension at most $16$. It remains to show that $\dim \Stab(F_8)\leq 2$, so let $(L,R)\in \Stab(F_8)$. We have $Cs(F_8)=\lra{x,y,x^2}$, hence $R=(\alpha x,\beta_1y+\beta_2x)$, where $\alpha,\beta_1\in\C^*$ and $\beta_2\in\C$. Moreover, $x^2$ and $xy$ do not have a linear part so $y$ and $\beta_1y+\beta_2x$ are proportional, i.e.$\beta_2=0$.
\end{proof}

\begin{theorem}\label{thf9}
Let $F=(x^2+y,xy,ax+by)$. If $a\in\C^*$ and $b=0$ then $F\in O(F_9)$, i.e. $F$ is linearly equivalent to $F_9=(x^2+y,xy,x)$. Moreover, $\dim O(F_9)=15$.
\end{theorem}

\begin{proof}
Obviously if $a\in\C^*$ then $(x^2+y,xy,ax)\in O(x^2+y,xy,x)$. To show that $\dim O(F_9)\leq 15$ notice that $O(F_9)$ is contained in the closure of $O(F_8)$ and $O(F_8)$ is open in its closure. Hence $\dim O(F_9)<\dim O(F_8)=16$. It remains to show that $\dim \Stab(F_9)\leq 3$, so let $(L,R)\in \Stab(F_9)$. We have $Cs(F_8)=\lra{1,x,2x^2-y}$, hence $R=(\alpha_1 x+\alpha_2,\beta_1x+\alpha_1^2y+\beta_2)$, where $\alpha_1\in\C^*$ and $\alpha_2,\beta_1,\beta_2\in\C$. Consider the second component of $F\circ R$, i.e. $(\alpha_1 x+\alpha_2)(\beta_1x+\alpha_1^2y+\beta_2)$. It contains the terms $\alpha_1\beta_1x^2+\alpha_1^2\alpha_2y$ which can be only canceled by $x^2+y$ from the first component. Consequently $\beta_1=\alpha_1\alpha_2$, so $\dim \Stab(F_9)\leq 3$.
\end{proof}

The last remaining case for $F=(x^2+y,xy,ax+by)$ is $a=b=0$, i.e $F=F_{10}$. We read from \cite{fj} that $\dim \Stab(x^2+y,xy)=1$ (with respect to $\mathcal{GA}(2,2)$), hence by Corollary \ref{codimo} we have $\dim \Stab(F_{10})=4$ and $\dim O(F_{10})=14$.

The third mapping on the list in \cite{fj} is $(x^2+y,y^2)$, thus we should consider mappings of the form $F=(x^2+y,y^2,ax+by)$. However if $a\neq 0$ then the critical set of $(x^2+y,y^2+ax+by)$ is a hyperbola, so by \cite{fj} $(x^2+y,y^2+ax+by)\in O(x^2+y,y^2+x)$. This case was considered earlier so we assume that $a=0$. That leaves us with two possibilities: either $b\neq 0$ and $F\in O(F_{11})$ or $b=0$ and $F=F_{12}$. In the former case we $Cs(F_{11})=\lra{x,xy}$, so if $(L,R)\in \Stab(F_{11})$ then $R=(\alpha_1 x,\beta_1y+\beta_2)$ for $\alpha_1,\beta_1\in\C^*$ and $\beta_2\in\C$. It is easy to see that for any such $R$ there exists a suitable $L$, hence $\dim \Stab(F_{11})= 3$. In the latter case we use the description of $(x^2+y,y^2)$ from \cite{fj} and obtain that $\dim \Stab(F_{12})=1+3=4$.

The fourth mapping on the list in \cite{fj} is $(x^2,y^2)$. Obviously if $a\neq 0$ or $b\neq 0$ then $(x^2,y^2,ax+by)$ is in $O(x^2+y,y^2,a'x+b'y)$ for some $a'$ and $b'$. Hence the only case not considered earlier is $F_{13}=(x^2,y^2,0)$. Since by \cite{fj} $\dim \Stab(x^2,y^2)=2$ we obtain $\dim \Stab(F_{13})=5$.

The fifth mapping on the list in \cite{fj} is $(x^2-x,xy)$, so let us consider mappings of the form $F=(x^2-x,xy,ax+by)$. If $b\neq 0$ then the critical set of $(x^2-x+ax+by,xy)$ is a parabola, so by \cite{fj} $(x^2-x+ax+by,xy)\in O(x^2+y,xy)$, which is a case considered earlier. Thus we have either $a\neq 0$ and $F\in O(F_{14})$ or $a=0$ and $F=F_{15}$. In the former case we $Cs(F_{14})=\lra{x,x^2}$ so if $(L,R)\in \Stab(F_{11})$ then $R=(\alpha_1 x,\beta_1y+\beta_2x+\beta_3)$ for $\alpha_1,\beta_1\in\C^*$ and $\beta_2,\beta_3\in\C$. It is easy to see that for any such $R$ there exists a suitable $L$, hence $\dim \Stab(F_{14})= 4$. In the latter case we use the description of $(x^2-x,xy)$ from \cite{fj} and obtain that $\dim \Stab(F_{15})=2+3=5$.

The sixth mapping on the list in \cite{fj} is $(x^2,xy)$. This is the final mapping on the list whose orbit does not intersect $\Omega_{\C^2}(2,1)$. Obviously if $a\neq 0$ or $b\neq 0$ then $(x^2,y^2,ax+by)$ is in one of the orbits described earlier. So, we only obtain $F_{16}=(x^2,xy,0)$ with $\dim \Stab(F_{16})=3+3=6$.

Since we exhausted all orbits which do not intersect $\Omega_{\C^2}(2,1,1)$ let us focus on the case when $F\in\Omega_{\C^2}(2,1,1)$, i.e $\dim_q(F)\leq 1$. If $\dim_a(F)=3$ then we may assume that $F=(f_1,x,y)$, where $f_1$ is homogeneous of degree $2$. There are two possibilities: either $f_1$ is a square or not. In the latter case we have $F\in O(F_{17})$, i.e. $F$ is linearly equivalent to $(xy,x,y)$, and in the former case we have $F\in O(F_{18})$, i.e. $F$ is linearly equivalent to $(x^2,x,y)$. It is easy to see that if $(L,R)\in \Stab(F_{17})$ then $R=(\alpha_1 x+\alpha_2,\beta_1y+\beta_2)$ for $\alpha_1,\beta_1\in\C^*$ and $\alpha_2,\beta_2\in\C$. Moreover for all such $R$ there is a unique $L$ such that $(L,R)\in \Stab(F_{17})$. Similarly for $(L,R)\in \Stab(F_{18})$ we obtain the condition $R=(\alpha_1 x+\alpha_2,\beta_1y+\beta_2x+\beta_3)$ for $\alpha_1,\beta_1\in\C^*$ and $\alpha_2,\beta_2,\beta_3\in\C$. Consequently $\dim \Stab(F_{17})=4$ and $\dim \Stab(F_{18})=5$.

Finally, the only case left to consider is when $\dim_a(F)\leq 2$, i.e. $O(F)\cap\Omega_{\C^2}(2,1,0)\neq\emptyset$. In that case we might assume that $F=(f_1,f_2,0)$, where $(f_1,f_2)$ is one of the $11$ mappings listed in \cite{fj} that belong to $\Omega_{\C^2}(2,1)$ (i.e. $f_7$ --- $f_{17}$ on the list). We add those mappings to our list as $F_{19}$ --- $F_{29}$. We calculate thee size of their orbits using Corollary \ref{codimo}.

\section{The real case}\label{secreal}

In this section we will determine the orbits in $\Omega_{\R^2}(2,2,2)$ over the field of real numbers. We will base on the complex case and to avoid confusion we will use the subscripts $\rr$ and $\cc$ to indicate over which field an object is considered. Obviously $O_{\rr}(F)\subset O_{\cc}(F)\cap \Omega_{\R^2}(2,2,2)$ for $F \in \Omega_{\R^2}(2,2,2)$. In most cases we have $O_{\rr}(F)= O_{\cc}(F)\cap \Omega_{\R^2}(2,2,2)$ however there are $7$ cases when $O_{\cc}(F)\cap \Omega_{\R^2}(2,2,2)=O_{\rr}(F)\cup O_{\rr}(F')$. We start with the following:

\begin{theorem}
The restriction of the complex orbit $O_{\cc}(F_1)$ splits into two real orbits. The first is represented by  $F_1=(x^2+y,y^2+x,xy)$ with $C(F_1)=\{(\frac{1}{2},\frac{1}{2})\}$ and $SI_{F_1}=\{(x,y):x+y-1=0\}$. The second case is represented by $F_{1'}=(x^2-y^2+x,2xy-y,-3x^2+y^2)$ with singular points $P_1=\left( \frac{1}{2},0\right)$, $P_2=\left(\frac{1}{4},\frac{\sqrt{3}}{4} \right)$, $P_3=\left(\frac{1}{4},\frac{-\sqrt{3}}{4} \right)$ and $SI_{F_{1'}}=\{(x,y):(x-\frac{1}{2})(x^2-\frac{1}{3}y^2)=0\}$.
\end{theorem}

\begin{proof}
Let $F\in O_{\C}(F_1)\cap \Omega_{\R^2}(2,2,2)$. Note that $F$ has three singular points over $\C$. Since the singular points are given by real equations they are either real or in complex conjugate pairs. Note that $F_{1'}$ has three real singular points and $F_1$ has one.

Assume that $F$ has three real singular points. There are some $L\in GA_{\C}(3)$, $R \in GA_{\C}(2)$ such that $F_{1'}=L \circ F \circ R$. Note that $R$ maps the three real and non-collinear singular points of $F_{1'}$ into the three real singular points of $F$. Hence $R(\R^2)=\R^2$, so $R\in GA_{\R}(2)$. Furthermore the affine space spanned by $F \circ R(\R^2)$ is $\R^3$, hence also $L\in GA_{\R}(3)$. Thus $F\in O_{\R}(F_{1'})$. 

Now assume that $F$ has one singular point. There are some $L\in GA_{\C}(3)$, $R \in GA_{\C}(2)$ such that $F_{1}=L \circ F \circ R$. By composing with an element of $\Stab (F_1)$ (which has one of following forms: $(x,y)$, $(\varepsilon x,\varepsilon^2 y)$, $(\varepsilon^2 x,\varepsilon y)$, $(y,x)$, $(\varepsilon y,\varepsilon^2 x)$, $(\varepsilon^2 y,\varepsilon x)$, $\varepsilon^3=1$) we may assume that $R$ maps the real singular point of $F_1$ to the real singular point of $F$.Since the complex singular points are conjugate the mapping $\overline{R}$, i.e. the conjugate of $R$, coincides with $R$ on the three non-collinear singular points. Hence $\overline{R}=R$, so $R\in GA_{\R}(2)$ and also $L\in GA_{\R}(3)$.
\end{proof}

\begin{theorem}
$O_{\R}(F_2) = O_{\C}(F_2)\cap \Omega_{\R^2}(2,2,2)$.
\end{theorem}

\begin{proof}
Let $F \in O_{\C}(F_2)\cap \Omega_{\R^2}(2,2,2)$. Recall from Theorem \ref{thf2} and its proof that $C(F_2)$ consists of a point and a double point and $SI_{F_2}$ consists of a line and a double line. The double line passes through both singular points and the ordinary line crosses only through the ordinary point. Thus $SI_F$ is given by a real polynomial of degree $3$ which decomposes into a square of a polynomial of degree $1$ and a second polynomial of degree $1$. In particular the components of this polynomial cannot be conjugate and thus have to be real.

Now let $R_0\in GA_{\R}(2)$ be such that $S_{F\circ R_0}=S_{F_2}$ and $C(F\circ R_0)=C(F_2)$. Moreover let $L\in GA_{\C}(3)$, $R \in GA_{\C}(2)$ be such that $F_2=L \circ F\circ R_0 \circ R$. Observe that $R$ must be identity on the double line of $SI_{F_2}$ and must preserve the ordinary line of $SI_{F_2}$ together with the critical point on it. According to the proof of Theorem \ref{thf2} that implies that there is some $L_1\in GA_{\C}(3)$ such that $(L_1,R)\in\Stab(F_2)$. It follows that $F_2=L_1^{-1}\circ L \circ F\circ R_0$ and since $\R^3$ is the affine span of $F\circ R_0(\R^2)$ we obtain that $L_1^{-1}\circ L\in GA_{\R}(3)$. Thus $F\in O_{\R}(F_2)$.
\end{proof}

\begin{theorem}
$O_{\R}(F_3) = O_{\C}(F_3)\cap \Omega_{\R^2}(2,2,2)$.
\end{theorem}

\begin{proof}
Recall that $Cs(F_3)=\lra{x^2,xy,2y^2-x}$. Let $(L,R)\in\mathcal{GA}_{\C}(2,3)$. Note that $Cs(L\circ F_3\circ R)=\lra{x_1^2,x_1y_1,ay_1^2-x_1}$, where $R(\{x_1=0\})=\{x=0\}$ and $R(\{y_1=0\})=\{y=0\}$ and $a\in\C^*$.

Now let $F \in O_{\C}(F_3)\cap \Omega_{\R}(2,2,2)$. We have $Cs(F)=\lra{x_2^2,x_2y_2,by_2^2-x_2}$, where $b\in\R^*$, $x_2=0$ is the equation of $SI_F$ (determined uniquely up to scalar multiplication) and $y_2=0$ is the equation of a line distinct from $\{x_2=0\}$ that passes through $C(F)$. Now we may find $R_0\in GA_{\R}(2)$ such that $Cs(F\circ R_0)=Cs(F_3)$. Moreover, there are $(L_1,R_1)\in\mathcal{GA}_{\C}(2,3)$ such that $F_3=L_1\circ F\circ R_0\circ R_1$. Observe that $R_1$ must preserve $Cs(F_3)$ and according to the proof of Theorem \ref{thf3} that implies that there is some $L_2\in GA_{\C}(3)$ such that $(L_2,R_1)\in\Stab(F_3)$. It follows that $F_3=L_2^{-1}\circ L_1 \circ F\circ R_0$ and since $\R^3$ is the affine span of $F\circ R_0(\R^2)$ we obtain that $L_2^{-1}\circ L_1\in GA_{\R}(3)$. Thus $F\in O_{\R}(F_3)$.
\end{proof}

\begin{theorem}
$O_{\R}(F_4) = O_{\C}(F_4)\cap \Omega_{\R^2}(2,2,2)$.
\end{theorem}

\begin{proof}
Let $F\in O_{\C}(F_4)\cap \Omega_{\R}(2,2,2)$. Take $R_0\in GA_{\R}(2)$ such that $C(F\circ R_0)=C(F_4)=(0,0)$ and $L_0\in GA_{\R}(3)$ be such that $L_0\circ F\circ R_0(0,0)=(0,0)$. Let $(f_1,f_2,f_3)=L_0\circ F\circ R_0$ and observe that $f_1$, $f_2$ and $f_3$ must be homogeneous of degree $2$, moreover, they must be linearly independent. Hence there is some $L_1\in\GL(\R^3)$ such that $L_1\circ(f_1,f_2,f_3)=(x^2,y^2,xy)$. Thus $F\in O_{\R}(F_4)$.
\end{proof}

\begin{theorem}
$O_{\R}(F_5) = O_{\C}(F_5)\cap \Omega_{\R^2}(2,2,2)$.
\end{theorem}

\begin{proof}
Let $F\in O_{\C}(F_5)\cap \Omega_{\R}(2,2,2)$. Note that $Cs(F)$ determines the singular point of $F$ and a pair of points at infinity. Let $R_0\in GA_{\R}(2)$ be such that $R_0(0,0)$ is the critical point of $F$ and $R_0$ sends the pair $\{(1:0),(0:1)\}$ to the pair determined by $Cs(F)$. Consequently we obtain $Cs(F\circ R_0)=\lra{x,y,xy}=Cs(F_5)$. We may additionally require that the line $SI_{F\circ R_0}$ is given by the equation $x+y=0$, i.e. that $SI_{F\circ R_0}=SI_{F_5}$. Now let $(L,R)\in\mathcal{GA}_{\C}(2,3)$ be such that $F_5=L\circ F\circ R_0\circ R$. Since $R$ must preserve the critical point $(0,0)$ and the pair $\{(1:0),(0:1)\}$ we must have either $R(x,y)=(ax,by)$ or $R(x,y)=(ay,bx)$ for some $a,b\in\C^*$. 
Without loss of generality we may assume it is the former. Moreover, since $R$ must also preserve the self-intersection curve $\{x+y=0\}$ we must have $a=b$. We have $F_5\circ R^{-1}=L_1\circ F_5$ for $L_1(x,y,z)=(a^{-2}x,a^{-2}y,a^{-1}x)$. Thus $L\circ L_1\circ F_5 = F\circ R_0$ and since $\R^3$ is the affine span of $F\circ R_0(\R^2)$ we obtain that $L\circ L_1\in GA_{\R}(3)$. Thus $F\in O_{\R}(F_3)$.
\end{proof}

\begin{theorem}
$O_{\R}(F_6) = O_{\C}(F_6)\cap \Omega_{\R^2}(2,2,2)$.
\end{theorem}

\begin{proof}
Let $F=(f_1,f_2,f_3)\in O_{\C}(F_6)\cap \Omega_{\R}(2,2,2)$. Since $F_6\in \Omega_{\R}(2,2,1)$ the homogeneous parts of degree $2$ of $f_1$, $f_2$ and $f_3$ are not linearly independent, hence there is an $L\in GA_{\R}(3)$ such that $L\circ F\in \Omega_{\R}(2,2,1)$. Moreover there are $a,b\in\C$ such that for $L_1=(x+az,y+bz,z)$ the set $V_{\C}(J_3(L_1\circ L\circ F))$ is a hyperbola. Consequently $V_{\C}(J_3(L_1\circ L\circ F))$ is a hyperbola for a generic choice of $a,b\in\C$ and also for a generic choice of $a,b\in\R$. Thus we obtain $L_1\circ L\circ F=(f_1,f_2,f_3)\in\Omega_{\R}(2,2,1)$ with $V(f_1,f_2)$ a hyperbola. By \cite{fj} there is $(L_2,R)\in\mathcal{GA}_{\R}(2,2)$ such that either $L_2\circ(f_1,f_2)\circ R=(x^2+y,y^2+x)$ or $L_2\circ(f_1,f_2)\circ R=(x^2-y^2+x,2xy-y)$. Let $L_3(x,y,z)=(L_2(x,y),z)$, we obtain two possibilities:

(1) $F'=L_3\circ L_1\circ L\circ F\circ R=(x^2+y,y^2+x,cx+dy)$ for some $c,d\in\R$. By Theorem \ref{thf6} either $c=0$ and $d\in\R^*$ or $d=0$ and $c\in\R^*$. In both cases $F'\in O_{\R}(F_6)$ hence $F\in O_{\R}(F_6)$.

(2) $F'=L_3\circ L_1\circ L\circ F\circ R=(x^2-y^2+x,2xy-y,cx+dy)$ for some $c,d\in\R$. If $c=d=0$ then $F'\in O_{\C}(F_7)$, a contradiction. Moreover, if $c^2+d^2\neq 0$ then $C(F')=\left(\frac{c^2-d^2}{2(c^2+d^2)},\frac{-cd}{2(c^2+d^2)}\right)$ so $F'\in O_{\C}(F_5)$, again a contradiction. Thus (2) is not possible and we have (1) and $F\in O_{\R}(F_6)$.
\end{proof}

\begin{re}
Observe that $O_{\R}(F_i) = O_{\C}(F_i)\cap \Omega_{\R^2}(2,2,2)$ for $i=8,9,10,11,12,14,15$. We have $O_{\R}(x^2+y,xy) = O_{\C}(x^2+y,xy)\cap \Omega_{\R^2}(2,2)$, $O_{\R}(x^2+y,y^2) = O_{\C}(x^2+y,y^2)\cap \Omega_{\R^2}(2,2)$ and $O_{\R}(x^2-x,xy) = O_{\C}(x^2-x,xy)\cap \Omega_{\R^2}(2,2)$ so the analysis we carried out in the complex case can be repeated in the real case.
\end{re}

\begin{re}
Observe that for $i=7,13,16,19,\ldots,29$ the mapping $F_i$ is a composition of a mapping in $\Omega_{\R^2}(2,2)$ and the standard inclusion $\R^2\rightarrow\R^3$. Thus from \cite{fj} we obtain that $O_{\R}(F_i) = O_{\C}(F_i)\cap \Omega_{\R^2}(2,2,2)$ for $i=16,20,\ldots,24,26,\ldots,29$ and $O_{\C}(F_i)\cap \Omega_{\R^2}(2,2,2)=O_{\R}(F_i)\cup O_{\R}(F_{i'})$ for $i=7,13,19,25$.
\end{re}

\begin{re}
Note that $F_{17}$ and $F_{18}$ represent the situation when $F=(f_1,x,y)$ and $f_1$ is homogeneous of degree $2$. Over $\C$ we have $F\in O(F_{17})$ if $f_1$ is not a square and $F\in O(F_{18})$ if $f_1$ is a square. Over $\R$ we additionally obtain $F\in O(F_{17'})$ if $f_1$ is neither a square nor a product.
\end{re}

\section{Higher dimensions}\label{secHD}

First note that for $n\geq 5$ the orbit $O(x^2,xy,y^2,x,y,0,\ldots,0)$ with respect to the $\mathcal{GA}(2,n)$ action is open and dense in the subspace $\Omega_{\C^2}(2,\ldots,2)$ of polynomial mappings $\C^2\to\C^n$. Its complement is the set $\mathcal{GA}(2,n)\Omega_{\C^2}(2,2,2,1,0,\ldots,0)$. Thus to fully classify all polynomial mappings $\C^2\to\C^n$ with degrees at most $2$ we only need to examine $\Omega_{\C^2}(2,2,2,1)$. Moreover, having classified all the orbits in $\Omega_{\C^2}(2,2,2)$ we can restrict ourselves to $\Omega_{\C^2}(2,2,2,1)\setminus\mathcal{GA}(2,4)\Omega_{\C^2}(2,2,2,0)$. This set consists of only $3$ orbits whose representatives are:

\begin{itemize}

\item[(1)] (generic case) $G_1=(x^2+y,y^2,xy,x)$ which is an immersion. Moreover, $\dim O(G_1)=24$.

\item[(2)] $G_2=(x^2,y^2,xy,x)$ with singular point $(0,0)$ and $SI_{G_2}=\{x=0\}$. Moreover, $\dim O(G_2)=23$.

\item[(3)] $G_3=(x^2,y^2,x,y)$ which is an immersion. Moreover, $\dim O(G_3)=22$.

\item[(4)] $G_4=(x^2,xy,x,y)$ which is an immersion. Moreover, $\dim O(G_4)=21$.

\end{itemize}

We will now provide a brief argument why those are the only orbits in $\Omega_{\C^2}(2,2,2,1)\setminus\mathcal{GA}(2,4)\Omega_{\C^2}(2,2,2,0)$. First assume that $$G=(g_1,g_2,g_3,g_4)\in \Omega_{\C^2}(2,2,2,1)\setminus\mathcal{GA}(2,4)\Omega_{\C^2}(2,2,1,1).$$ Let $x=g_4$ and $y$ be linearly independent from $x$. Since $G\notin\mathcal{GA}(2,4)\Omega_{\C^2}(2,2,1,1)$ the quadratic parts of $g_1$, $g_2$ and $g_3$ are independent, so we may assume that they are $x^2$, $y^2$ and $xy$. So $G=(x^2+ay,y^2+by,xy+cy,x)$ for some $a,b,c\in\C$. Taking $R(x,y)=(x-c,y-\frac{b}{2})$ and $L(x,y,z,w)=(x+2cw+c^2+\frac{ab}{2},y+\frac{b^2}{4},z+\frac{bw+bc}{2},w+c)$ we obtain $L\circ G\circ R=(x^2+ay,y^2,xy,x)$. So if $a=0$ then $G\in O(G_2)$ and if $a\neq 0$ then by taking $R'(x,y)=(ax,ay)$ and $L'(x,y,z,w)=(a^{-2}x,a^{-2}y,a^{-2}z,a^{-1}w)$ we obtain $G\in O(G_1)$. Moreover, the orbit $O(G_2)$ is contained in the closure of the orbit $O(G_1)$ because $(x^2+y/n, y^2, xy, x)\in O(G_1)$ for every $n$. Furthermore, $O(G_1)\cup O(G_2)=\{G\in \Omega_{\C^2}(2,2,2,2):\ \dim_q(G)\geq 3\}$ is open. Consequently the orbit $O(G_1)$ is open, hence $\dim O(G_1)=24$.

To see that $\dim O(G_2)=23$ observe that if $(L,R)\in\Stab(G_2)$ then $R$ preserves $C(G_2)$ and $SI_{G_2}$ so $R(x,y)=(\alpha x,\beta_1 x+\beta_2 y)$ for $\alpha,\beta_2\in\C^*$ and $\beta_1\in\C$. Moreover for each such $R$ there is precisely one $L$ such that $(L,R)\in\Stab(G_2)$. Hence $\dim\Stab(G_2)=3$ and $\dim O(G_2)=23$.

Now assume that $G=(g_1,g_2,g_3,g_4)\in \Omega_{\C^2}(2,2,1,1)\setminus\mathcal{GA}(2,4)\Omega_{\C^2}(2,2,2,0)$. Since $g_3$ and $g_4$ are linear and linearly independent we may choose them to be any basis of $\lra{x,y}$. Furthermore by subtracting suitable combinations of $g_3$ and $g_4$ we may assume that $g_1$ and $g_2$ are homogeneous of degree $2$. In particular the set $C(g_1,g_2)$ and $\Delta(g_1,g_2)$ are  curves which are cones. Checking the list in \cite{fj} we obtain that $(g_1,g_2)\in O(x^2,y^2)$ or $(g_1, g_2)\in O(x^2, xy)$ are the only possibilities. Thus $G\in O(G_3)$ or $G\in O(G_4).$ Observe also that if $(L,R)\in\Stab(G_3)$ then $R$ preserves the vector space $\lra{1,x,y,x^2,y^2}$ so $R(x,y)=(\alpha_1x+\alpha_2,\beta_1y+\beta_2)$ for $\alpha_1,\beta_1\in\C^*$ and $\alpha_2,\beta_2\in\C$. Moreover for each such $R$ there is precisely one $L$ such that $(L,R)\in\Stab(G_2)$. Hence $\dim\Stab(G_3)=4$ and $\dim O(G_3)=22$. 

Similarly, if $(L,R)\in\Stab(G_4)$ then $R$ preserves the vector space $\lra{1,x,y,x^2,xy}$ so $R(x,y)=(\alpha_1x+\alpha_2,\beta_1y+\beta_2x+\beta_3)$ for $\alpha_1,\beta_1\in\C^*$ and $\alpha_2,\beta_2,\beta_3\in\C$. Moreover for each such $R$ there is precisely one $L$ such that $(L,R)\in\Stab(G_4)$, hence $\dim\Stab(G_4)=5$ and $\dim O(G_4)=21.$

Finally, let us note that the analysis above is also correct for real polynomial mappings, with one exception. The list in \cite{fj} with real mappings contains three homogeneous polynomials of degree $(2,2)$, namely $(x^2,y^2)$, $(x^2-y^2,xy)$ and $(x^2,xy)$. For the first two mappings the components do not have a common factor, for the third mapping they do. Consequently $O_{\C}(G_{3})\cap \Omega_{\R^2}(2,2,2,2)$ consists of two orbits: $O_{\R}(G_{3})$ and $O_{\R}(G_{3'})$ for $G_{3'}=(x^2-y^2,xy,x,y)$. Moreover, $O_{\C}(G_{4})\cap \Omega_{\R^2}(2,2,2,2)=O_{\R}(G_{4})$.

\section{Containment of closures of orbits}\label{secOCC}

In this section we show for which orbits the smaller is contained in the closure of the larger. The results are presented graphically in Figure  \ref{Picture1}. We omit most of the trivial cases and the cases when mappings in both orbits have $\dim_a$ at most $2$, which are covered in the paper \cite{fj}. When containment does hold we simply provide a sequence contained in one orbit that converges to a point in the other (with respect to the standard norm on $\C^N$). When containment does not hold we assume that there is a mapping in the larger orbit that is close to a mapping in the smaller orbit and arrive at a contradiction. We were tempted to present a presumably more elegant method - to give an equation for the larger orbit and check that it does not hold for a mapping from the smaller orbit, however the equations in question are quite large and not elegant at all.

$O(G_2)\ni(x^2,y^2,\frac{xy}{n}+y,x)\to(x^2,y^2,y,x)\in O(G_3)$

$O(G_3)\ni(x^2-\frac{y^2}{n},xy,x,y)\to(x^2,xy,x,y)\in O(G_4)$

Observe that the quadratic parts of components of mappings in $O(G_4)$ have a common factor, since the same does not hold for $O(F_4)$ we see that $O(F_{4})\not\subset \overline{O(G_4)}$. Similarly $O(F_{11})$, $O(F_7)$ and $O(F_{12})$ are not contained in $\overline{O(G_4)}$, $O(F_{12})$ and $O(F_{13})$ are not contained in $\overline{O(F_9)}$ and $O(F_{13})\not\subset \overline{O(F_{14})}$.

$O(F_2)\ni(x^2+y,y^2+x,\frac{axy}{n}+x+\frac{y}{n})\to(x^2+y,y^2+x,x)\in O(F_6)$ for some $a$ close to $\frac{1}{2n}$ satisfying $32(1-\frac{1}{n^3})a=\frac{16}{n}-\frac{72}{n^2}a^2-\frac{27}{n^3}a^4$

$O(F_2)\ni(x^2,\frac{y^2}{n}+\frac{x}{n}+y,xy)\to(x^2,y,xy)\in O(F_8)$

$O(F_5)\ni(x^2,\frac{y^2}{n}+xy,y)\to(x^2,xy,y)\in O(F_8)$

Now we will show that $O(F_{11})\not\subset \overline{O(F_3)}$. To this end we will show, that there is no $F\in O(F_3)$ which is close to $F_{11}$. Suppose that $F=(f_1,f_2,f_3)\in O(F_3)$ with $f_i=a_ix^2+b_ixy+c_iy^2+d_ix+e_iy+g_i$ is close to $F_{11}$. By taking $L\circ F$ for suitable $L$ we may additionally assume that $a_1=1$ and also $a_2=a_3=0$. Similarly we may assume that also $c_2=1$, $c_1=c_3=0$, $e_3=1$ and $e_1=e_2=0$, so $F=(x^2+b_1xy+e_1x,y^2+b_2xy+e_2x,y+b_3xy+e_3x)$. Note that since $F\in O(F_3)$ we must have $b_3\neq 0$. Thus $F$ is equivalent to $F'=(x^2+Ay,y^2+Bx,xy+Cx+Dy)$ for $b_3A=-b_1$, $b_3B=b_3e_2-b_2e_3$, $2b_3C=2e_3+b_2$ and $2b_3D=2+b_1e_3-b_3e_1$. From Theorem \ref{thf3} we know that $C(F')=V(4xy-AB,2x^2+2Dx-Ay-AC,Bx+BD-2y^2-2Cy)$ is a triple point. Hence if $AB=0$ then $C=D=0$, in particular $b_3e_1-b_1e_3=2$ which is impossible since $b_1,b_3,e_1,e_3$ are close to $0$. Thus $AB\neq 0$ and $F$ is equivalent to $(x^2+y,y^2+x,xy+C(AB^2)^{-\frac{1}{3}}x+D(A^2B)^{-\frac{1}{3}}y)$. From Theorem \ref{thf3} we obtain that $-8D^3=27A^2B$ which is a contradiction since $-64$ is not close to $0$.

$O(F_3)\ni(x^2+y-\frac{3y^2}{n^2},\frac{y^2}{n^3}+x,2xy+\frac{3y^2}{n})\to(x^2+y,x,2xy)\in O(F_9)$

$O(F_3)\ni(x^2+y,y^2+x,\frac{xy}{n}+\frac{3x}{2n}+\frac{3y}{2n})\to(x^2+y,y^2+x,0)\in O(F_7)$

$O(F_6)\ni(x^2+y,\frac{y^2}{n}+xy,x)\to(x^2+y,xy,x)\in O(F_9)$

$O(F_9)\ni(x^2+\frac{y}{n},xy,x)\to(x^2,xy,x)\in O(F_{14})$

$O(F_9)\ni(\frac{x^2}{n}+y,xy,x)\to(y,xy,x)\in O(F_{17})$

$O(F_{11})\ni(x^2,\frac{y^2}{n}+xy,x)\to(x^2,xy,x)\in O(F_{14})$

To show that $O(F_{17})\not\subset \overline{O(F_{11})}$ suppose that $F\in O(F_{11})$ is close to $F_{17}$. We may assume that $F=(xy+a_1x^2+c_1y^2,x+a_2x^2+c_2y^2,y+a_3x^2+c_3y^2)$. Since $\dim_q(F)=\dim_q(F_{11})=2$ we have $a_2c_3-a_3c_2=0$ and $J_1(F)=(1+2a_2x)(1+2c_3y)-4a_3c_2xy=1+2a_2x+2c_3y$. So $C(F)=V(1+2a_2x+2c_3y)$, in particular $J_1(F)$ must divide $J_3(F)=2c_2y^2+4(a_1c_2-a_2c_1)xy-2a_2x^2-x-2c_1y$. Since $J_3(F)$ does not have a constant term we must have $J_3(F)=-(1+2a_2x+2c_3y)(x+2c_1y)$, so $4(a_1c_2-a_2c_1)=-4a_2c_1-2c_3$ and $2c_2=-4c_1c_3$. This implies $c_3(1-4a_1c_1)=0$. Since $a_1$ and $c_1$ are close to $0$ we have $c_3=c_2=0$. Similarly $J_1(F)$ must divide $J_2(F)$ which implies $a_2=a_3=0$. However this means that $\dim_q(F)=1$, which is a contradiction.

We show that $O(F_{10})\not\subset \overline{O(F_{11})}$ similarly as with $O(F_{17})$ and $O(F_{11})$. Indeed, if $F$ is close to $F_{11}$ then $J_3(F)$ is close to $2x^2-y$, in particular it is an irreducible quadric. Thus $C(F)$ may not be a line, so $F\notin O(F_{11})$.

$O(F_4)\ni(x^2,\frac{y^2}{2n}+y,\frac{xy}{n}+x)\to(x^2,y,x)\in O(F_{18})$

$O(F_{14})\ni(x^2,\frac{xy}{n}+y,x)\to(x^2,y,x)\in O(F_{18})$

Now we will show that $O(F_{20})\not\subset \overline{O(F_4)}$, since obviously $O(F_{20})\subset \overline{O(F_{15})}$ and $O(F_{20})\subset \overline{O(F_{19})}$ that implies that also $O(F_{15})$ and $O(F_{19})$ are not contained in the closure of $O(F_4)$. Suppose that $F\in O(F_4)$ is close to $F_{20}$. We may assume that $F=(xy+a_1x^2+c_1y^2+e_1y,x+a_2x^2+c_2y^2+e_2y,a_3x^2+c_3y^2+e_3y)$. Since $\dim_q(F)=\dim_q(F_4)=3$ we have $t=a_2c_3-a_3c_2\neq 0$. Thus $F$ is equivalent to $F'=(xy+a_1x^2+c_1y^2+e_1y,x^2+t^{-1}(c_3x+(c_3e_2-c_2e_3)y),y^2+t^{-1}(-a_3x+(-a_3e_2+a_2e_3)y))$. Since $F'\in O(F_4)$ we must have $c_3e_2-c_2e_3=a_3=0$, so $t=a_2c_3\neq 0$ and $F'=(xy+a_1x^2+c_1y^2+e_1y,x^2+t^{-1}c_3x,y^2+t^{-1}a_2e_3y)$. Moreover, $F'$ is equivalent to $F''=(xy+Ax+By,x^2,y^2)$ for $2tA=-2a_1c_3-a_2e_3$ and $2tB=2te_1-2a_2c_1e_3-c_3$. Again, since $F''\in O(F_4)$ we must have $A=B=0$ which leads to $1=4a_1c_1+2a_2e_1$ which is a contradiction with $a_1$, $c_1$, $e_1$ and $a_2$ being close to $0$.

To show that $O(F_{19})\not\subset \overline{O(F_{14})}$ suppose that $F\in O(F_{14})$ is close to $F_{19}$. We may assume that $F=(xy+a_1x^2+c_1y^2+e_1y,x+y+a_2x^2+c_2y^2,a_3x^2+c_3y^2+e_3y)$. First observe that since $F\in O(F_{14})$ we have $\dim_q(F)=2$ (so $a_2c_3=a_3c_2$) and $\dim_a(F)=3$, moreover, if one of the components of $F$ is homogeneous of degree $1$ then that component must divide the quadratic parts of the other two components. Now we will show that $c_3\neq 0$. Indeed if $a_3=c_3=0$ then $e_3\neq 0$ and $y$ must divide $a_2x^2+c_2y^2$. In particular $a_2=0$, so $J_1(F)=e_3\neq 0$ and $C(F)=\emptyset$, a contradiction. Furthermore if $c_3=0$ and $a_3\neq 0$ then $F$ is equivalent to $F'=(xy+a_1x^2+c_1y^2+e_1y,x+(1-a_3^{-1}a_2e^3)y,a_3x^2+e_3y)$. Since $x+(1-a_3^{-1}a_2e_3)y$ must divide $x^2$ we have $a_3=a_2e_3$, but then $J_1(F')=e_3\neq 0$ gives a contradiction again. Thus we must have $c_3\neq 0$. Let $t=1-c_3^{-1}c_2e_3$, then $F$ is equivalent to $(xy+a_1x^2+c_1y^2+e_1y,x+ty,a_3x^2+c_3y^2+e_3y)$. Since $x+ty$ must divide $a_3x^2+c_3y^2$ we obtain $a_3\neq 0$ and $t\neq 0$. Furthermore $F$ is equivalent to $(a_1x^2+(1-2a_1t)xy+(c_1-t+a_1t^2)y^2+e_1y,x,a_3x^2-2a_3txy+(c_3+a_3t^2)y^2+e_3y)$ and we obtain the conditions $c_1-t+a_1t^2=0$, $c_3+a_3t^2=0$ and $(1-2a_1t)e_3=-2a_3te_1$ which will lead to a contradiction. Indeed, we have $t^{-1}c_1+a_1t=1$, since $a_1$ and $c_1$ are close to $0$ it implies that either $t$ is close to $c_1$ or $t$ is close to $a_1^{-1}$. In the former case it follows that $c_3$ is close to $c_2e_3$, $a_3$ is close to $a_2e_3$ and $e_3$ is close to $2te_1a_3$. Thus $2te_1a_2$ is close to $1$, which is a contradiction. In the latter case it follows that $c_3$ is close to $-a_1c_2e_3$, $a_3$ is close to $-c_3a_1^2$ and $a_1e_3$ is close to $2a_3e_1$. Thus $2a_1^2c_2e_1$ is close to $1$, which is the final contradiction.

\section{Topological types}\label{top}

Now we briefly describe topological types of mappings in $\Omega_{\C^2}(2,\ldots,2)$. The case $n=1$ is trivial and the case $n=2$ is described in \cite{fj}. Hence we start with the case $n=3$. We have following 18 topological types:

\begin{itemize}

\item[(1)] $F_1$.
\item[(2)] $F_2$.
\item[(3)] $F_3$.
\item[(4)] $F_4$.
\item[(5)] $F_5$, $F_8$, $G_2$.
\item[(6)] $F_6$, $F_9$, $F_{17}$, $F_{18}$, $F_{22}$, $F_{24}$, $G_0$, $G_1$, $G_3$, $G_4$.
\item[(7)] $F_{7}$.
\item[(8)] $F_{10}$.
\item[(9)] $F_{11}$, $F_{19}$, $F_{21}$.
\item[(10)] $F_{12}$.
\item[(11)] $F_{13}$.
\item[(12)] $F_{14}$, $F_{20}$.
\item[(13)] $F_{15}$.
\item[(14)] $F_{16}$.
\item[(15)] $F_{23}$, $F_{26}$, $F_{28}$.
\item[(16)] $F_{25}$.
\item[(17)] $F_{27}$.
\item[(18)] $F_{29}$.
\end{itemize}

In dimension $n=4$ there are no new topological types, indeed, $G_1$, $G_3$, $G_4$ are of type (6) and $G_2$ is of type (5). Similarly $G_0$ is of type (6), hence in dimensions $n>3$ there are no new topological types. 

In most cases showing equivalence or the lack of equivalence is trivial. However there are two exceptions which we will now handle. We will show that  $F_5$ is topologically equivalent to $F_8$ and that the mappings $F_3$ and $F_5$ are not equivalent (i.e. that the cases (3) and (5) are different). We have
$$F_5\circ(x-y,y)=(x^2,y^2,x+y)\circ (x-y,y)=(x^2-2xy+y^2,y^2,x)\ \rm{and}$$
$$(-\frac{1}{2}(x-y-z^2),y,z)\circ (x^2-2xy+y^2,y^2,x)=(xy,y^2,x)\sim (x^2,xy,y)=F_8.$$

Now assume that $F_3$ is topologically equivalent to $F_5$. Let $P=0$ describe the surface $F_3(\C^2)$ and $Q=0$ describe the surface $F_5(\C^2)$. Since $F_3$ is topologically equivalent with $F_5$ the germ $P_0$ is topologically equivalent to the germ $Q_0$. Now we use the following result of Navarro Aznar (see \cite{r-t}):

\vspace{3mm}

{\it If $P,Q:(\C^3,0)\to (\C,0)$ are reduced germs of holomorphic functions 
 which are topologically equivalent and $\mu_0(Q)=2$ then also $\mu_0(P)=2$.}

\vspace{3mm} 

Now observe that in our case $P=x^3-x^2y^2+2xyz^2-z^4$, $Q=x^2-2xy-y^2-2xz^2-2yz^2+z^4$, hence $\mu_0(Q)=2$ but $\mu_0(P)=3$ and we obtain a contradiction. Thus $F_3$ is not topologically equivalent to $F_5$. 

It is worth to note that for quadratic mappings from $\C^2$ to $\C^n$ topological equivalence coincides with equivalence given by the group of global polynomial automorphisms of $\C^2$ and $\C^n$.

\begin{re}
{\rm In \cite{fj} we stated the following:}

{\bf Conjecture 1}. {\it For every $d_1, d_2>0$ the set $U$ of topologically generic mappings in $\Omega_{\C^2}(d_1,d_2)$ is an open affine subvariety of $\Omega_{\C^2}(d_1,d_2)$. In particular every topologically generic mapping is topologically stable, i.e. remains topologically generic after a small deformation.}

{\rm Note that from our topological classification we have that this conjecture is not true for $\Omega_{\C^2}(d_1,\ldots,d_k)$, where $k\geq 4$. Indeed, $F(x,y)=(x,y,0,0)$ is topologically generic in $\Omega_{\C^2}(2,2,2,2)$.
Now take $F_n(x,y)=\left(\frac{y^2}{n}+x,\frac{x^2}{n}+y,0,0\right),$ none of $F_n$ is topologically generic however $\displaystyle \lim_{n\rightarrow\infty}F_n=F$. The same argument works for all spaces $\Omega_{\C^2}(2,2,2,\ldots,2)$.

However it is worth to mention that our conjecture works for the space $\Omega_{\C^2}(2,2,2)$.}
\end{re}

\vspace{5mm}

Similarly in the real case we have the following types (here we use the group of real polynomial automorphisms of $\R^2$ and $\R^n$):

\begin{itemize}

\item[(1)] $F_{1}$.
\item[(2)] $F_{1'}$.
\item[(3)] $F_2$.
\item[(4)] $F_3$.
\item[(5)] $F_4$.
\item[(6)] $F_5$, $F_8$, $G_2$.
\item[(7)] $F_6$, $F_9$, $F_{17}$, $F_{17'}$, $F_{18}$, $F_{22}$, $F_{24}$, $G_0$, $G_1$, $G_3$, $G_{3'}$.
\item[(8)] $F_{7}$.
\item[(9)] $F_{7'}$.
\item[(10)] $F_{10}$.
\item[(11)] $F_{11}$, $F_{19}$, $F_{19'}$, $F_{21}$.
\item[(12)] $F_{12}$.
\item[(13)] $F_{13}$.
\item[(14)] $F_{13'}$.
\item[(15)] $F_{14}$, $F_{20}$.
\item[(16)] $F_{15}$.
\item[(17)] $F_{16}$.
\item[(18)] $F_{23}$, $F_{26}$, $F_{28}$.
\item[(19)] $F_{25}$.
\item[(20)] $F_{25'}$.
\item[(21)] $F_{27}$.
\item[(22)] $F_{29}$.

\end{itemize}

In dimension $n=4$ there are no new topological types, indeed, $G_1$, $G_3$, $G_{3'}$, $G_4$ are of type (7) and $G_2$ is of type (6). Similarly $G_0$ is of type (7), hence in dimensions $n>3$ there are no new topological types.


\begin{thebibliography}{99}

\bibitem{agrach} A. A. Agrachev, R. V. Gamkerlidze, {\em Quadratic maps and smooth vector valued functions: Euler characteristic of the level sets}, Itogi Nauki i Tekchniki, vol. 35 (1989), 179-239.

\bibitem{a-n} K. Aoki, H. Noguchi, {\em On topological types of polynomial map germs of plane to plane}, Mem. School
Sci. Eng. Waseda Univ. 44 (1980), 133-156.

\bibitem{bgv} F. Bofill, J. Garrido, F. Vilamaj{\'o}, N. Romero, A. Rovella, {\em On the Quadratic Endomorphisms of the Plane}, Advanced Nonlinear Studies, 4 (2004), 37-55.

\bibitem{dgr} J. Delgado, J. L. Garrido, N. Romero, A. Rovella, F.  Vilamajo, {\em On the Geometry of Quadratic Maps of the Plane},
Publ. Mat. Uruguay, 14 (2013), 120-135.

\bibitem{fj} M. Farnik, Z. Jelonek, \emph{On quadratic polynomial mappings of the plane}, Linear Algebra Appl. vol. 529 (2017), 441-456

\bibitem{gio} G. Giorgadze, {\em Quadratic mappings and configuration spaces}, Banach Center Publications, vol. 62 (2004), 73-86.

\bibitem{i-n} S. Ichiki, T. Nishimura, \emph{Distance-squared mappings},  
Topol. Appl. 160 (2013), 1005-1016.

\bibitem{in2} S. Ichiki, T. Nishimura, \emph{Recognizable classification of
Lorentzian distance-squared mappings}, {J. Geom. Phys.} 81 (2014), 62-71.

\bibitem{in3} S. Ichiki, T. Nishimura, R. Oset Sinha, M. A. S. Ruas,
\emph{Generalized distance-squared mappings of the plane into the plane}, {Adv. Geom.} 16(2) (2016), 189-198.

\bibitem{in4} S. Ichiki, T. Nishimura, \emph{Generalized distance-squared mappings of $\mathbb{R}^{n+1}$ into $\mathbb{R}^{2n+1}$}, Contemporary Mathematics, vol. 675, Amer. Math. Soc., Providence RI, (2016), 121-132.

\bibitem{jel2} Z. Jelonek, {\em On semi-equivalence of generically-finite polynomial
mappings},   Math. Z. 283(1)  (2016), 133-142.

\bibitem{npg} C.-H. Nien, B. B. Peckham, R. P. McGehee, {\em Classification of critical sets and their images for quadratic maps of the plane},
J. Differ. Equations Appl. 22(5) (2016), 637-655.

\bibitem{r-t} J. Risler, D. Trotman, {\em Bilipschitz invariant of the multiplicity}, Bull. London Math. Soc. 29 (1997), 200-204.

\bibitem{sabb} C. Sabbah, {\em Le types topologique {\'e}clacte d'une application analytique}, Singularities, Part 2 (Arcata, Calif., 1981) Proc. Sympos. Pure Math., vol. 40, Amer. Math. Soc., Providence, RI, 1983, 433-440.

\end{thebibliography}
\end{document}